\documentclass[12pt]{amsart}

\usepackage{graphicx,amsthm}
\usepackage{enumerate}
\usepackage{empheq}
\usepackage{amssymb}
\usepackage[margin=1in]{geometry}
\usepackage{pst-all}
\usepackage{array}
\usepackage{tikz}
\usepackage{tikz-cd}
\usetikzlibrary{arrows,positioning,shapes.geometric}
\usepackage{amsmath,amscd,amsfonts,bbm,mathabx,graphicx}
\allowdisplaybreaks
\usepackage[onehalfspacing]{setspace}

\usepackage{verbatim} 

\usepackage[all]{xy} 

\usepackage{caption}
\usepackage{subcaption}

\usepackage[mathscr]{euscript}
\usepackage{booktabs} 
\usepackage{ amssymb }
\usepackage{hyperref}

\usepackage{soul}

\newtheorem{theorem}{Theorem}[section]
\newtheorem{lemma}[theorem]{Lemma}
\newtheorem{proposition}[theorem]{Proposition}

\newtheorem{problem}[theorem]{Question}

\theoremstyle{definition}
\newtheorem{definition}[theorem]{Definition}
\newtheorem{example}[theorem]{Example}

\newtheorem{remark}[theorem]{Remark}
\numberwithin{equation}{section}

\newtheorem{thmy}{Theorem}
\newenvironment{thmx}{\stepcounter{theorem}\begin{thmy}}{\end{thmy}}

\newcommand{\C}{\mathbb{C}}

\DeclareMathOperator{\GL}{GL}

\newcommand{\flag}[1]{\mathcal{F}\ell(\mathbb{C}^{#1})} 
\newcommand{\Xwo}[1]{X_{#1}^{\circ}}
\newcommand{\Xw}[1]{X_{#1}}
\newcommand{\Owo}[1]{\Omega_{#1}^{\circ}}
\newcommand{\Ow}[1]{\Omega_{#1}}

\newcommand{\Owh}[1]{\Omega_{#1,h}}

\DeclareMathOperator{\Hess}{Hess}

\newcommand{\Gh}{\Gamma_h}
\DeclareMathOperator{\codim}{\codim}

\newcommand{\Sn}[1]{\mathfrak{S}_{#1}} 
\newcommand{\ve}{\mathbf{e}}
\newcommand{\hpat}[1]{{#1}_h} 
\newcommand{\gen}{\mathcal{G}_h}

\allowdisplaybreaks

\definecolor{cadmiumgreen}{rgb}{0.0, 0.42, 0.24}

\begin{document}

\title[Characterization of the smoothness of Hessenberg Schubert varieties]{Towards combinatorial characterization of the smoothness of Hessenberg Schubert varieties}

\author{Soojin Cho}
\address{Department of Mathematics, Ajou University, Suwon  16499, Republic of Korea}
\email{chosj@ajou.ac.kr}

\author{JiSun Huh}
\address{Department of Mathematics, University of Seoul, Seoul, Republic of Korea}
\email{hyunyjia@yonsei.ac.kr}

\author{Seonjeong Park}
\address{Department of Mathematics Education, Jeonju University, Jeonju 55069, Republic of Korea}
\email{seonjeongpark@jj.ac.kr}

\thanks{This work was supported by the National Research Foundation of Korea [NRF-2020R1A2C1A01011045].}

\begin{abstract}  
A \emph{Hessenberg Schubert variety} is an irreducible component of the intersection of a Schubert variety and a Hessenberg variety, defined as the closure of a Schubert cell intersected with the Hessenberg variety. 
We consider the smoothness of Hessenberg Schubert varieties of regular semisimple Hessenberg varieties of type $A$ in this paper. 

We consider the smoothness of the intersection of a Schubert variety and a Hessenberg variety to ensure the smoothness of the corresponding Hessenberg Schubert variety. Specifically, we analyze the structure of the GKM graphs of the intersection of a Schubert variety 
 and a Hessenberg variety. Our results show that
the regularity of these GKM graphs is completely characterized in terms of pattern avoidance, which is a necessary and sufficient
condition for the intersection to be smooth. This shows that our pattern avoidance provides a sufficient condition for the smoothness of a Hessenberg Schubert variety.
\end{abstract}


\maketitle

\setcounter{tocdepth}{1}

\section{Introduction} \label{sec:intro}

Schubert varieties in full flag varieties form an important class of complex projective varieties that appear in many areas, including algebraic geometry, representation theory, and combinatorics. Schubert varieties of the full flag variety of type $A_{n-1}$ are subvarieties of $\GL_n(\mathbb C)/B$, where $B$ is the Borel subgroup of upper triangular matrices in $\GL_n(\mathbb C)$. Each Schubert variety is defined as the closure of a Schubert cell,
\[
  X_w \coloneqq \overline{BwB/B},
\]
and the index \(w\) runs over the symmetric group \(\mathfrak{S}_n\), the Weyl group of \(\GL_n(\mathbb{C})\).
The closure of the opposite Schubert cell $\Omega_w^\circ\coloneqq B^-wB/B$ is the opposite Schubert variety $\Omega_w$ of $\GL_n(\mathbb C)/B$, where $B^-$ is the subgroup of lower triangular matrices in $\GL_n(\mathbb C)$.
The characterization of the smoothness of $X_w$ (hence of $\Omega_w$) was done in different ways: pattern avoidance, the regularity of the Bruhat subgraph, the Poincar\'e polynomial of the cohomology, and the Kazhdan--Lusztig polynomial (see Section 13.2 of \cite{BL}). We write some of them, which are relevant to our work, in the following theorem. Note that a maximal torus $T\subset B$ acts on the flag variety $\GL_n(\mathbb{C})/B$ by left multiplication, and the $T$-fixed points can be naturally identified with the elements of $\mathfrak{S}_n$. Furthermore, the opposite Schubert varieties are $T$-stable subvarieties of $\GL_n(\mathbb{C})/B$.

Recall that the GKM graph of the flag variety has vertices corresponding to the \(T\)-fixed points, which can be naturally identified with elements of the symmetric group. Moreover, it is isomorphic to the Bruhat graph of \(\mathfrak{S}_n\).

\begin{theorem}\label{thm:known_smoothness}\cite{BL,Carrell,LSa} The following statements are equivalent:
\begin{enumerate}
\item The opposite Schubert variety $\Omega_w$ is smooth.
\item The subgraph of the Bruhat graph of $\mathfrak S_n$ (equivalently, the GKM graph of $\GL_n(\mathbb C)/B$) induced by the set $\Omega_w^T$ of torus fixed points is a regular graph.
\item The permutation $w$ avoids patterns $2143$ and $1324$.
\end{enumerate}
\end{theorem}

Hessenberg varieties are subvarieties of the full flag variety, which were introduced by De Mari, Shayman, and Proceci \cite{DPS} in the 1990s. 
A diagonal matrix $S$ with distinct eigenvalues and a nondecreasing function $h\colon \{1, \dots, n\} \to \{1, \dots, n\}$ with $h(i)\geq i$ for $i=1, \dots, n$, define a regular semisimple Hessenberg variety $\Hess(S, h)$ of type $A$. Owing to their interesting characteristics, Hessenberg varieties have become one of the central objects of research in related areas, particularly combinatorics and algebraic geometry. A surprising instance is that the well-known Stanley--Stembridge conjecture in algebraic combinatorics is shown to be equivalent to a conjecture on the $\mathfrak{S}_n$-module structure of the cohomology of the corresponding  Hessenberg variety $\Hess(S, h)$ \cite{BC,G-P}. In this context, Cho, Hong, and Lee~\cite{CHL} considered the minus cell decomposition (the Bia{\l}ynicki--Birula decomposition)
\[
\bigsqcup_{w \in \mathfrak{S}_n} \bigl(\Omega_w^\circ \cap \Hess(S,h)\bigr)
\]
of \(\Hess(S,h)\) and the cohomology classes \(\sigma_{w,h}\) of \(\overline{\Omega_w^\circ \cap \Hess(S,h)}\) to investigate the \(\mathfrak{S}_n\)-module structure of the cohomology space.
Then
\[
\Omega_{w,h}^\circ \coloneqq \Omega_w^\circ \cap \Hess(S,h)
\quad\text{and}\quad
\Omega_{w,h} \coloneqq \overline{\Omega_{w,h}^\circ}
\]
are called the \emph{opposite Hessenberg Schubert cell} and the \emph{opposite Hessenberg Schubert variety}, respectively. In the same paper, they showed that the support of the class $\sigma_{w, h}$ is the set  $\Omega_{w, h}^T$  of torus fixed points of the Hessenberg Schubert variety and provided an explicit combinatorial description of the torus fixed points. The permutations in $\mathfrak{S}_n=\Hess(S,h)^T$ are partitioned according to the Hessenberg function~$h$. More precisely, for a Hessenberg function~$h$, we define
\begin{equation*}
    \gen\coloneqq \{w\in\mathfrak S_n\mid w^{-1}(w(j)+1)\leq h(j)\text{ for all }w(j)\leq n-1\}.
\end{equation*}
Then for every $w\in\Sn{n}$, there is a unique element $\widetilde{w}\in\gen$ 
satisfying, for all $i<j$ with $j\leq h(i)$,
\begin{equation}\label{eq:generator}
    \widetilde{w}(i)<\widetilde{w}(j)\qquad\text{if and only if}\qquad w(i)<w(j).
\end{equation}
See also~\cite[Proposition~3.8]{HP}.
The cohomology classes corresponding to $w\in\gen$ generate the $\mathfrak{S}_n$-module  $H^\ast(\Hess(S,h))$; see  Section~3 of \cite{CHL2}. 
In particular, we have 
\[\Omega_{w,h}=u(\overline{\Omega_{\widetilde{w}}^\circ\cap\Hess(u^{-1}Su,h)}),\] 
where $u=w\widetilde{w}^{-1}$, see~\cite[Proposition~3.14]{HP}.

Little is known about the geometry and topology of Hessenberg Schubert varieties; see~\cite{CHL, HP} for more details.  
Recently, Insko, Precup, and Woo~\cite{IPW} investigated the smoothness of Hessenberg Schubert varieties defined as the closures of the intersections between a regular Hessenberg variety and Schubert cells associated with the Hessenberg function $h=(2,3,\dots,n,n)$. 
In contrast, our focus is on the smoothness of Hessenberg Schubert varieties defined as the intersections between a regular semisimple Hessenberg variety and Schubert cells for an arbitrary Hessenberg function $h$. 
In our setting, when $h=(2,3,\dots,n,n)$, the regular semisimple Hessenberg variety is the permutohedral variety, which is known to be a smooth toric variety \cite{DPS}; hence, every Hessenberg Schubert variety in this case is smooth.

Since a Hessenberg Schubert variety $\Ow{w,h}$ is an irreducible component of $\Ow{w}\cap \Hess(S,h)$, if $\Ow{w}\cap \Hess(S,h)$ is smooth, then $\Owh{w}$ is also smooth, but the converse is not true~\cite[Example~6.5]{HLP}.
In this paper, we aim to extend the results (in Theorem~\ref{thm:known_smoothness}) regarding the characterization of the smoothness of the Schubert varieties $\Omega_w$ to those for the Hessenberg Schubert varieties $\Omega_{w, h}$ by considering the smoothness of the intersections $\Omega_{w}\cap\Hess(S,h)$. 

The Hessenberg variety $\Hess(S, h)$ is known to be a GKM variety, and its GKM graph (denoted by $\Gamma_h$) is a subgraph of the GKM graph of the full flag variety, with the same set $\mathfrak S_n$ of vertices~\cite{DPS}. We define the \emph{$h$-Bruhat order} as the transitive closure of the following relation:
\[u\prec_h v\text{ if and only if } u\prec v\text{ and there is an edge connecting $u$ and $v$ in $\Gamma_h$}.\]
We first 
interpret the set $\Omega_{w, h}^T$, which has been shown in \cite{HP} to be the interval $[w, w_0]$ in the Bruhat order for $w\in \gen$, where $w_0$ denotes the longest element, in terms of the $h$-Bruhat order. We use $[n]$ for the set $\{1, 2, \dots, n\}$.

\begin{thmx}[Theorem~\ref{thm:interval}]\label{thmx:1}
For a given Hessenberg function $h\colon [n] \to  [n]$, if $w\in\gen$, then $\Omega_{w, h}^T=[w, w_0]_h$, the interval between $w$ and $w_0$ in the $h$-Bruhat order. 
\end{thmx}

Let $\Gamma_{w,h}$ be the subgraph of $\Gamma_h$ induced by the set $\Omega_{w, h}^T$ of torus fixed points. As stated in Proposition~\ref{prop:Hessenberg Schubert intersection}, if $\Ow{w}\cap\Hess(S,h)$ is smooth and connected, then $w\in\gen$ and $\Gamma_{w,h}$ is regular. Furthermore, if $w$ is the representative in $\gen$ corresponding to $v\in\Sn{n}$, then $\Gamma_{v,h}$ and $\Gamma_{w,h}$ are isomorphic by Lemma~\ref{lemm:iso}. In the following theorem, we show that for $w\in\gen$, the degree of the vertices in $\Gamma_{w,h}$ is nondecreasing with respect to the $h$-Bruhat order. This theorem is essential for characterizing the regularity of $\Gamma_{w, h}$.  
For two permutations $u$ and $v$, we use $u\preceq_h v$ to denote the $h$-Bruhat order relation.

\begin{thmx}[Theorem~\ref{thm:increasing}]\label{thmx:2}
For a Hessenberg function $h$ and $w\in \gen$, if $u$ and $v$ are vertices in $\Gamma_{w,h}$ with $u\preceq_h v$, then $\deg(u)\leq \deg(v)$, where the degree of a vertex is the number of edges that are incident to the vertex.
\end{thmx}

If we apply Theorem~\ref{thmx:2} to $h=(n,\dots,n)$, we obtain that if $u\prec v$ in the Bruhat order, 
then $\deg(u)\leq \deg(v)$ in the Bruhat graph $\Gamma_w$ for every $w\in \mathfrak{S}_n$. 
Consequently, $\Gamma_w$ is regular if and only if $\deg(w)=\deg(w_0)$ in $\Gamma_w$. 
This, in turn, recovers the well-known result that $\Omega_w$ is smooth if and only if it is smooth at $w_0$ 
(see~\cite[p.~208]{BL}).

We then consider the regularity of $\Gamma_{w,h}$ in relation with patterns in the permutation~$w$, where the patterns depend upon the given Hessenberg functions~$h$. We find all the patterns that $w$ must avoid for $\Gamma_{w,h}$ to be regular as stated in the following theorem. The patterns are defined in Definitions~\ref{def:pattern4} and~\ref{def:pattern5}. 

\begin{thmx}[Theorems \ref{thm:irregular}, \ref{thm:regular},  and \ref{thm:main}]\label{thmx:3}
Let $h\colon [n] \to [n]$ be a Hessenberg function.

\begin{enumerate}
    \item For $w\in \gen$, the graph $\Gamma_{w,h}$ is regular if and only if $w$ avoids all the associated patterns $\hpat{2143}$, $\hpat{1324}$, $\hpat{1243}$, $\hpat{2134}$, $\hpat{1423}$, $\hpat{2314}$, and $\hpat{2413}$. 

    \item For $w\in \mathfrak S_n$, the graph $\Gamma_{w,h}$ is regular if and only if $w$ avoids all the associated patterns $\hpat{2143}$, $\hpat{1324}$, $\hpat{1243}$, $\hpat{2134}$, $\hpat{1423}$, $\hpat{2314}$, $\hpat{25314}$, $\hpat{24315}$, $\hpat{14325}$, and $\hpat{15324}$. 
\end{enumerate}
\end{thmx}

When \(h = (n,n,\dots,n)\), i.e, $\Hess(S,h)=\flag{n}$, every permutation belongs to \(\gen\) and automatically avoids the patterns
\(\hpat{1243}\), \(\hpat{2134}\), \(\hpat{1423}\), \(\hpat{2314}\), and \(\hpat{2413}\).
Therefore, \(\Gamma_{w,h}\) is regular if and only if \(w\) avoids the patterns
\(\hpat{2143}\) and \(\hpat{1324}\).
This agrees with Theorem~\ref{thm:known_smoothness}.

Recently, Hong, Lee, and Park~\cite{HLP} proved that for $w\in\gen$, the intersection $\Omega_w\cap\Hess(S,h)$ is smooth if and only if $\Gamma_{w,h}$ is regular. Therefore, our pattern avoidance condition gives a sufficient condition for the smoothness of $\Owh{w}$ for $w\in\Sn{n}$ (see Theorems~\ref{conj:regular},~\ref{thm:final}, and Figure~\ref{fig:diagram}).
 
 This paper is organized as follows. In Section~\ref{sec:2}, first, we set up notation and terminology, and then review some basic theories on Schubert varieties and Hessenberg Schubert varieties. In Section~\ref{sec:GKM}, we focus on the GKM graph $\Gamma_h$ of a Hessenberg variety and its subgraph $\Gamma_{w, h}$. The set of vertices of $\Gamma_{w, h}$ is characterized in terms of the $h$-Bruhat order, and a nice injective map from the set of edges of a vertex to the set of edges of an incident vertex is introduced. Further, the seven patterns
that permutations $w\in\gen$ must avoid in order for the variety $\Ow{w}\cap\Hess(S,h)$ to be smooth are also introduced in Section~\ref{sec:GKM}. In Section~\ref{sec:regular}, we prove that avoiding the seven patterns of length~$4$ for permutations in $\gen$ and ten patterns of length $4$ or~$5$ for arbitrary permutations suffices to ensure that the corresponding graph $\Gamma_{w, h}$ is regular. 
Hence, we show that avoiding these patterns suffices for the smoothness of $\Owh{w}$ for $w\in\Sn{n}$.

\section{Preliminaries}\label{sec:2}

\subsection{Basic terminologies and properties.}
Let $\mathfrak{S}_n$ be a symmetric group on $[n]$. For a permutation $w\in\mathfrak{S}_n$, we use the one-line notation 
$w=w(1)w(2)\cdots w(n).$ For $1\leq i<j\leq n$, the permutation that acts on $[n]$ by swapping $i$ and $j$ is called a \emph{transposition}, and it is denoted by~$(i,j)$. The \emph{Bruhat order} on $\mathfrak{S}_n$ is the transitive closure of the relation
$$u\prec v \quad \text{ if and only if } \quad v=u(i,j) \text{ for some $(i,  j)$ and } \ell(u)<\ell(v)\,.$$
The \emph{length} of $w$ is defined by the number of inversions and denoted by $\ell(w)$, i.e.,
$\ell(w)=\left|\{ (i,j)\mid w(i,j) \prec w \}  \right|$.
The poset $(\Sn{n},\prec)$ is a graded poset whose rank function is given by the length of a permutation.
Notably, $(\Sn{n},\prec)$ has the unique minimal element $e = 1\, 2\,\cdots\,n$ and the unique maximal element $w_0=n\,(n-1)\,\cdots\,1.$  The \emph{Bruhat interval} $[v,w]$ is the subposet of $(\Sn{n},\prec)$ defined by $[v,w]=\{u\in \Sn{n}\mid v\preceq u\preceq w\}$.
The \emph{Bruhat graph} for $\Sn{n}$ is the graph with the vertex set $\Sn{n}$ and edges $\{v,w\}$ if $w=v(i,j)$ for some $(i,j)$. 

For $w \in \Sn{n}$ and $p \in \Sn{k}$ with $k \leq n$, we say that the permutation $w$ \emph{contains} the pattern $p$ if a sequence $1 \leq i_1 < \cdots < i_k \leq n$ exists such that $w(i_1)\cdots w(i_k)$ has the same relative order as $p(1)\cdots p(k)$. If $w$ does not contain $p$, then we say that $w$ \emph{avoids} $p$ or is \emph{$p$-avoiding}. 

For positive integers $k_1\leq k_2$, we denote the set $\{ k_1, k_1+1, \dots, k_2 \}$ by $[k_1, k_2]$.
For a permutation $w\in \mathfrak{S}_n$,
we let $w[k_1,k_2]$ be the set $\{w(i)\mid k_1\leq i \leq k_2\}$.
We use $A\!\!\uparrow$ to represent the increasing sequence $\{a_1 < a_2 < \cdots < a_k \}$ of the elements in $A$, where $A$ is a set of $k$ integers. 
For two finite sets $A$ and $B$ of integers of the same cardinality, $A\!\!\uparrow \leq B\!\!\uparrow$ means $a_i\leq b_i$ for all $i$, where $a_i$ and $b_i$ are the $i$th elements of $A\!\!\uparrow$ and $B\!\!\uparrow$, respectively.

There are many equivalent conditions to $u\prec v$ in the Bruhat order; we give some of them that we use in this paper.
We refer to the book~\cite{BB} for more general theories on the Bruhat orders for Coxeter groups.  

\begin{proposition}\cite[Theorem 2.6.3]{BB}\label{prop:Bruhat order}
    For $u, v\in \mathfrak{S}_n$, the following statements are equivalent:
    \begin{enumerate}
        \item $u\preceq v$ in the Bruhat order.
        \item $u[k]\!\!\uparrow\leq v[k]\!\!\uparrow$ for all $k\in [n]$.\item $u[k]\!\!\uparrow\leq v[k]\!\!\uparrow$ for all $k\in [n-1]-D(v)$, where $D(v)\coloneqq \{ i\in [n-1]\mid v(i)>v(i+1)\}$ is the set of \emph{descents} of $v$.
    \end{enumerate}
\end{proposition}

\begin{lemma}\cite[Lemma 2.1]{Bil}\label{lemma:Bruhat order}
    Suppose two permutations $u$ and $v$ agree everywhere except on positions $i_1<\cdots<i_k$. Then $u\preceq v$ in the Bruhat order if and only if $\{ u(i_1), \dots, u(i_j)\}\!\!\uparrow \leq \{ v(i_1), \dots, v(i_j)\}\!\!\uparrow $ for all $j=1, \dots, k$. In particular, 
    $u\prec u(i,j)$ if and only if $u(i)<u(j)$.
\end{lemma}

The Bruhat order on $\Sn{n}$ satisfies the chain property as follows.

\begin{proposition}[Chain Property]\label{prop:chain}
    If $u\prec v$ in $\mathfrak S_n$, then there exist elements $v_i\in \mathfrak S_n$, satisfying $\ell(v_i)=\ell(u)+i$ for $0\leq i\leq k$, and $u=v_0\prec v_1\prec \cdots \prec v_k=v$.
\end{proposition}

\subsection{Flag varieties, Schubert varieties, and Opposite Schubert varieties.} 

The flag variety $\flag{n}$ is the homogeneous space $\GL_n(\mathbb{C})/B$, where $B$ is the Borel subgroup of upper triangular matrices in $\GL_n(\mathbb{C})$. The flag variety is a smooth projective variety of (complex) dimension $\binom{n}{2}$, and it can be identified with the set
\[\flag{n}\coloneqq\{(\{0\}\subset V_1 \subset V_2\subset\cdots \subset V_n=\C^n)\mid \dim_{\C}V_i=i\text{ for }i=1,\dots,n\}\]
of chains of subspaces of $\C^n$. Each element of $\flag{n}$ is called a flag.
For instance, an element $w \in \Sn{n}$ defines a flag given by
$$(\{0\}\subset\langle \ve_{w(1)}\rangle \subset \langle \ve_{w(1)},\ve_{w(2)}\rangle\subset\cdots\subset V_n =\C^n),$$
where $\ve_1,\dots,\ve_n$ are the standard basis vectors in $\C^n$. We denote by $wB$ this standard coordinate flag.

Let $G=\GL_n(\C)$ and $B^-$ be the Borel subgroup of lower triangular matrices in $G$.
Then, the left action of $B$ (respectively, $B^-$) on $G/B$ has finitely many orbits $BwB/B$ (respectively, $B^- wB/B$), where $w$ is a permutation in $\mathfrak{S}_n$, and we get a cell decomposition called the \emph{Bruhat decomposition}:
\begin{equation}\label{eq:flag-decomposition}
    G/B= \bigsqcup_{w \in \mathfrak{S}_n} BwB/B\,= \bigsqcup_{w \in \mathfrak{S}_n} B^- wB/B\,.
\end{equation}
For each $w\in\mathfrak{S}_n$, the cells $BwB/B$ and $B^-wB/B$ are isomorphic to $\C^{\ell(w)}$ and $\C^{\binom{n}{2}-\ell(w)}$, respectively. We call $\Xwo{w}\coloneqq BwB/B$ the \emph{Schubert cell} and $\Owo{w}\coloneqq B^-wB/B$ the \emph{opposite Schubert cell} indexed by~$w$. The (Zariski) closures $\Xw{w}\coloneqq \overline{\Xwo{w}}$ and $\Ow{w}\coloneqq \overline{\Owo{w}}$ are called the \emph{Schubert variety} and the \emph{opposite Schubert variety} indexed by~$w$, respectively. Notably, $\Ow{w}=w_0\Xw{w_0w}$ because $B^{-}=w_0Bw_0$, and we focus our attention on opposite Schubert varieties throughout the paper.

For $v,w\in\Sn{n}$, we have
\[v\prec w\quad \text{ if and only if }\quad \Ow{w} \subset \Ow{v},\]
and we obtain a cell decomposition of $\Ow{w}$ as follows:
\begin{equation}\label{eq:Schubert-decomposition}
    \Ow{w}= \bigsqcup_{v\succeq w} \Owo{v}\,.
\end{equation}
An opposite Schubert variety is not necessarily smooth. As stated in Theorem~\ref{thm:known_smoothness}, $\Ow{w}$ is smooth if and only if $w$ avoids patterns $2143$ and $1324$. \footnote{Note that $\Ow{w}=w_0X_{w_0w}$ and $w_0$ is an isomorphism. By \cite{LSa}, a Schubert variety $X_w$ is smooth if and only if $w$ avoids $3412$ and $4231$. Hence $\Ow{w}$ is smooth if and only if $w$ avoids $2143$ and $1324$.}

We refer the reader to \cite{Brion2005,Fulton} and the references therein for the geometry and combinatorics related to flag varieties and their subvarieties.

\subsection{Torus actions and GKM varieties.}
Let $X$ be a complex projective algebraic variety with an action of an algebraic torus $T\cong (\C^\ast)^n$. Then, the variety $X$ is called a \emph{GKM (Goresky--Kottwitz--MacPherson) space} if it satisfies the following three conditions:
\begin{itemize}
    \item The fixed point set $X^T$ consists of isolated points.
    \item There are finitely many one-dimensional orbits of $T$ on $X$.
    \item The space $X$ is equivariantly formal with respect to the action of $T$.
\end{itemize}
Note that the condition of being equivariantly formal is rather technical.
However, \eqref{eq:flag-decomposition} and \eqref{eq:Schubert-decomposition} imply that the odd-degree cohomology groups of \(\flag{n}\) and \(\Omega_w\) vanish, and hence both are equivariantly formal with respect to any torus action \cite[Theorem~14.1]{GKM}. We refer the reader to \cite{GKM, GZ, GT} for more on GKM theory.

For a GKM variety $X$, the boundary of each one-dimensional $T$-orbit contains two $T$-fixed points and the closure of each one-dimensional orbit is isomorphic to $\C P^1$ with fixed points at the north and the south poles. Based on the zero- and one-dimensional orbits of a GKM variety, we construct the GKM graph $\Gamma$ of $X$ as follows:
\begin{enumerate}
    \item the vertex set of $\Gamma$ is identified with the set $X^T$ of $T$-fixed points; and
    \item two vertices $p$ and $q$ of $X^T$ are connected by an edge if there exists a one-dimensional orbit whose closure has $p$ and $q$ as the $T$-fixed points.
\end{enumerate}

Now, we let $T$ be the set of diagonal matrices in $G=\GL_n(\mathbb{C})$. Then $T$ is isomorphic to $(\C^\ast)^n$ and acts on $G/B$ by left multiplication, and the set of $T$-fixed points of $G/B$ consists of the standard coordinate flags. Thus, we identify $(G/B)^T$ with $\Sn{n}$. Furthermore, the flag variety $\flag{n}$ with the action of $T$ becomes a GKM space, and the GKM graph $\Gamma$ of $\flag{n}$ is the Bruhat graph for $\Sn{n}$.

For every $w\in \mathfrak{S}_n$, the opposite Schubert variety $\Ow{w}$ is a GKM subspace of $\flag{n}$. The set of $T$-fixed points of $\Ow{w}$ is identified with the Bruhat interval $[w,w_0]$ and the GKM graph of $\Ow{w}$ is the subgraph of $\Gamma$ induced by $[w,w_0]$. As stated in Theorem~\ref{thm:known_smoothness}, $\Ow{w}$ is smooth if and only if the subgraph of $\Gamma$ induced by $[w,w_0]$ is regular.\footnote{Note that this holds only for Schubert varieties in flag varieties of simply laced reductive algebraic groups of type $A$, $D$, or $E$. In general, a Schubert variety is rationally smooth if and only if its GKM graph is regular \cite{Carrell}. However, for a non-simply laced type, not every rationally smooth Schubert variety is smooth. For example, $X_{s_1s_2s_1}$ is rationally smooth but not smooth in type~$C_3$. See~\cite[Chapter~6]{BL}.}

\subsection{Hessenberg varieties.}

Hessenberg varieties were introduced by De Mari, Shayman, and Procesi~\cite{DPS} in the study of the geometry of the adjoint representation of $\GL_n(\mathbb{C})$.  
They have since been extensively studied from various perspectives; see \cite{Tymoczko2006,AHMMS,Precup2} for further developments.

A \emph{Hessenberg function} is a nondecreasing function $h\colon [n] \to [n]$ satisfying $h(i)\geq i$ for each $i\in [n]$. We often use the list of values $(h(1), h(2), \dots, h(n))$ of $h$ to represent a Hessenberg function $h\colon [n] \to [n]$. Moreover, we can depict a Hessenberg function using a graph. For each Hessenberg function $h$, the graph with the vertex set $[n]$ and the edge set $\{\{i,j\} \mid 1\leq i<j\leq h(i) \}$ is called the \emph{incomparability graph} $\mathrm{inc}(h)$ of $h$.\footnote{The incomparability graph will be used later when introducing the associated patterns in Definitions~\ref{def:pattern4} and~\ref{def:pattern5}.} See Figure~\ref{fig:incomparability} for example.

\begin{figure}[ht]
    \centering
    \begin{tikzpicture}
        \foreach \i/\j in {1/i,2/j,3/k,4/\ell}{
		\coordinate (\i) at (-5+\i,0);
		\filldraw[color=black] (\i) circle (2pt);
		\node[below] at (\i) {$\i$};}
		\draw[thick] (1) -- (4);
		\draw[thick, bend left=80] (1) edge (3);
		\draw[thick, bend left=80] (2) edge (4);
    \end{tikzpicture}
    \caption{The incomparability graph $\mathrm{inc}(h)$  of $h=(3,4,4,4)$}
    \label{fig:incomparability}
\end{figure}

Let $S$ be a regular semisimple linear operator on $\C^n$. That is, a Jordan form of $S$ is a diagonal matrix with $n$ distinct eigenvalues. A regular semisimple \emph{Hessenberg variety} determined by $h$ and $S$ is defined as 
$$\Hess(S, h)\coloneqq \{(\{0\}\subset V_1 \subset \cdots \subset V_n) \in \flag{n} \,\, |\,\, S(V_i)\subseteq V_{h(i)} \mbox{ for } i=1, \dots, n \} \,.$$ Then $\Hess(S,h)$ is a smooth projective variety of (complex) dimension $d_h\coloneqq \sum_i (h(i)-i)$, see~\cite[Theorem~6]{DPS}.

\begin{remark}\label{remark:Hessenberg function}
\begin{enumerate}
    \item When $h=(n, n, \dots, n)$,  $\Hess(S, h)$ is the flag variety $\flag{n}$.
    \item If we associate a Hessenberg function $h$ with the space
    \[H(h)=\{(a_{i,j})_{1\leq i,j\leq n} \in \mathrm{Mat}_{n\times n}(\mathbb{C})\mid a_{i,j}=0\text{ for }i>h(j)\},\] then we can get an alternative definition of Hessenberg varieties given in~\cite{DPS}: \[\Hess(S,h)=\{gB\in \GL_n(\C)/B\mid g^{-1}Sg\in H(h)\}.\] From this definition, it follows easily that for any $g\in\GL_n(\mathbb{C}),$ we have \[\Hess(g^{-1}Sg,h)=g^{-1}\Hess(S,h),\] which is isomorphic to $\Hess(S,h)$.
	\item A regular semisimple Hessenberg variety $\Hess(S,h)$ is connected if and only if $h(i)\ge i+1$ for all $i<n$, or equivalently, if the incomparability graph $\mathrm{inc}(h)$ is connected. If $h(i)=i$ for some $i$, then the connected components of $\Hess(S,h)$ are smaller Hessenberg varieties. See~\cite{Precup2} for more general cases.
\end{enumerate}
\end{remark}
From now on, we assume that $S$ is a diagonal matrix with distinct eigenvalues. 
It was established in~\cite[Section~III]{DPS} that the Hessenberg variety \(\Hess(S,h)\) admits an affine cell decomposition:
\begin{equation}\label{eq:Hess_paving}
\Hess(S,h)
= \bigsqcup_{w \in \mathfrak{S}_n} \bigl(\Omega_w^{\circ} \cap \Hess(S,h)\bigr).
\end{equation}
For each $w\in\Sn{n}$, the intersection $\Omega_{w, h}^\circ \coloneqq  \Omega^{\circ}_w \cap \Hess(S, h)$ is isomorphic to $\C^{d_h-\ell_h(w)}$, and it is called the \emph{opposite Hessenberg Schubert cell} indexed by $w$, where $$\ell_h(w)= \left|  \{ (i,j) \mid i<j,\,\,w(i)> w(j), \,\, j\leq h(i)   \}  \right|\,.$$

In this paper, we focus on regular semisimple Hessenberg varieties and opposite Hessenberg Schubert cells. For simplicity, we omit the term `opposite' when we deal with (opposite) Hessenberg Schubert cells $\Ow{w,h}^\circ$.

\begin{proposition}[\cite{T2}] A regular semisimple Hessenberg variety $ \Hess(S, h)$  is a $T$-stable subvariety of $\flag{n}$, and it is a GKM variety with the associated GKM graph $\Gamma_h=(V, E)$, where $V= \mathfrak{S}_n$ and $E=\{ \{u, v\} \mid v=u(i,j) \mbox{ for }1\leq  i<j\leq h(i) \}$.
\end{proposition}

Note that $\Gamma_h$ is independent of the choice of $S$.

The closure $\Omega_{w, h}\coloneqq \overline{\Omega_{w, h}^\circ}=\overline{\Omega_{w}^\circ \cap \Hess(S,h)}$, is called the \emph{(opposite) Hessenberg Schubert variety} indexed by~$w$. Then $\Ow{w,h}$ is contained in $\overline{\Owo{w}}\cap\overline{\Hess(S,h)}=\Ow{w}\cap\Hess(S,h)$, so it is an irreducible component of 
\begin{equation}\label{eq:cell-decomp}
    \Ow{w}\cap \Hess(S,h)= \bigsqcup_{v \succeq w}\Owo{v,h}\,.
\end{equation}
Since $\Ow{w}$ and $\Hess(S,h)$ are $T$-stable subvarieties of $\flag{n}$ and $\Ow{w,h}$ is irreducible, $\Ow{w,h}$ is also $T$-stable. In the following, we characterize $\Ow{w,h}^T$.

\begin{proposition}[\cite{CHL}, \cite{CHL2}, \cite{HP}]\label{prop:h-fixed points} Let $h$ be a Hessenberg function on $[n]$. For a permutation $w\in \mathfrak S_n$, let $\widetilde{w}$ be the unique permutation in $\gen$ satisfying~\eqref{eq:generator}. Then the following statements hold.
\begin{enumerate}
    \item If $w\in\gen$, then $\Omega_{w, h}^T$ is identified with the Bruhat interval $[w, w_0]$. 
    \item If $w\not\in\gen$, then $\Omega_{w, h}^T$ is identified with the set $w\widetilde{w}^{-1}[\widetilde{w}, w_0]$.
\end{enumerate}
In particular, $\Ow{w,h}^T$ and $\Ow{\widetilde{w},h}^T$ have the same cardinality.
\end{proposition}

It should be noted that if $w\not\in\gen$, then $\Ow{w,h}^T$ does not form an interval. 

\begin{example}
    When $h=(3, 3, 4, 4)$, we get  
$$\gen=\{1234,1423,2134,2341,2431,3241,3412,3421,4123,4231,4312,4321\}.$$ Note that $w=1324\not\in\gen$ because $h(2)=3<w^{-1}(4)=4$, and $\widetilde{w}=1423$ is the unique permutation~in $\gen$ satisfying~\eqref{eq:generator}. In this case, $w_0\in \Ow{w,h}^T$, but the permutation $1432\not\in\Ow{w,h}^T$ even though 
$w\prec 1432 \prec w_0$.
\end{example}

The smoothness of the intersection of a Schubert variety indexed by $w\in\gen$ and the Hessenberg variety implies the smoothness of the Hessenberg Schubert varieties indexed by the permutations in the same class in the following way. Recall that we fix a diagonal matrix $S$ with distinct eigenvalues and the Hessenberg Schubert varieties $\Ow{w,h}$ are defined inside $\Hess(S,h)$.

\begin{proposition}\label{prop:intersection_smooth} For $w\in\Sn{n}$,  let $\widetilde{w}$ be the unique permutation in $\gen$ satisfying~\eqref{eq:generator}. Then the following statements hold:
\begin{enumerate}
    \item If \(\Omega_{\widetilde{w}} \cap \Hess(S,h)\) is smooth, then so is \(\Omega_{\widetilde{w},h}\). In particular, if \(\Omega_{\widetilde{w}} \cap \Hess(S,h)\) is connected, then
\(
\Omega_{\widetilde{w}} \cap \Hess(S,h) = \Omega_{\widetilde{w},h}.
\)
    \item If $\Omega_{\widetilde{w}}\cap \Hess(u^{-1}Su,h)$ is smooth, then $\Omega_{w,h}$ is smooth, where $u=w\widetilde{w}^{-1}$.
\end{enumerate}
\end{proposition}
\begin{proof}
\begin{enumerate}
    \item Note that $\Owh{\widetilde{w}}$ is an irreducible component of $\Omega_{\widetilde{w}}\cap\Hess(S,h)$. If $\Omega_{\widetilde{w}}\cap\Hess(S,h)$ is smooth, then every connected component is smooth and irreducible, so $\Owh{\widetilde{w}}$ is smooth.
    \item It is shown in \cite[Proposition 3.14]{HP}, \cite[Lemma 4.5]{CHL}, and \cite[Proposition~3.8]{CHL2} that $\Owh{w}=u(\overline{\Ow{\widetilde{w}}^\circ\cap \Hess(u^{-1}Su,h)})$ where $u=w\widetilde{w}^{-1}$. Note that $\Owh{\widetilde{w}}=\overline{\Ow{\widetilde{w}}^\circ\cap \Hess(S,h)}$ and $u^{-1}Su$ is obtained from $S$ by permuting the diagonal entries. If $\Omega_{\widetilde{w}}\cap \Hess(u^{-1}Su,h)$ is smooth, then $\overline{\Ow{\widetilde{w}}^\circ\cap \Hess(u^{-1}Su,h)}$ is smooth by (1). Hence, $\Owh{w}$ is also smooth because the left multiplication by~$u$ preserves the smoothness.
\end{enumerate}
\end{proof}

We next describe several properties of elements $w\in\gen$.
\begin{lemma}\label{lemm:gen_irr}
If $w\not\in\gen$, the intersection $\Ow{w}\cap\Hess(S,h)$ is reducible.
\end{lemma}
\begin{proof}
Note that $$(\Ow{w}\cap \Hess(S,h))^T=\Omega_w^T\cap\Hess(S,h)^T=\Omega_w^T=[w,w_0]$$ since $\Hess(S,h)^T=\mathfrak{S}_n$. By \cite[Lemma~2.11]{HP}, if $\widetilde{w}$ is the element $\gen$ corresponding to $w\in\mathfrak{S}_n$, then $w\leq\widetilde{w}$ in the Bruhat order. It follows from Proposition~\ref{prop:h-fixed points} that $\Ow{w,h}^T = [w,w_0]$ if and only if $w\in\gen$. Hence, if $w\not\in\gen$, we obtain $$\Omega_{w,h}^T=w\widetilde{w}^{-1}[\widetilde{w},w_0]\subsetneq [w,w_0]= (\Omega_w\cap \Hess(S,h))^T,$$ so $\Omega_{w,h}\subsetneq \Omega_w\cap\Hess(S,h)$. Therefore, the intersection $\Ow{w} \cap \Hess(S,h)$ is reducible. 
\end{proof}
Note that $\Ow{w}\cap\Hess(S,h)$ can be reducible even if $w\in\gen$. For instance, when  $h=(3,3,4,4)$ and $w=2134\in\gen$, the intersection $\Ow{w}\cap\Hess(S,h)$ is connected but reducible, see~\cite[Example~6.5]{HLP}. In this case, the corresponding opposite Hessenberg Schubert variety $\Ow{w,h}$ is smooth.

We let $\Gamma_{w,h}$ be the subgraph of $\Gh$ induced by the torus fixed point set $\Owh{w}^T$. Then for $w\in\gen$, $\Gamma_{w,h}$ becomes the GKM graph of  $\Ow{w}\cap \Hess(\overline{S},h)$ by~\eqref{eq:cell-decomp} for any diagonal matrix~$\overline{S}$ with distinct eigenvalues. As we will see below, $\Ow{w,h}^T$ and $\Ow{\widetilde{w},h}^T$ induce the isomorphic subgraphs of $\Gamma_h$.

\begin{lemma}\label{lemm:iso}
    For $w\in\Sn{n}$, $\Gamma_{w,h}$ is isomorphic to $\Gamma_{\widetilde{w},h}$.
\end{lemma}
\begin{proof}
    For each $w\in \Sn{n}$, by the definition of $\Gamma_{w,h}$, there is an edge $\{u,v\}\in E(\Gamma_{w,h})$ if and only if there exists a transposition $(i,j)$ such that $u(i,j)=v$ and $1\leq i<j\leq h(i)$. If $w\not\in\gen$, then $\Ow{w,h}^T=\{w\widetilde{w}^{-1}u\mid u\in [\widetilde{w},w_0]\}$ by Proposition~\ref{prop:h-fixed points}. Therefore, for $u'=w\widetilde{w}^{-1}u$ and $v'=w\widetilde{w}^{-1}v$ in $\Ow{w,h}^T$, $\{u',v'\}$ is an edge of $\Gamma_{w,h}$ if and only if $\{u,v\}$ is an edge of $\Gamma_{\widetilde{w},h}$.
\end{proof}

As an example, Figure~\ref{fig:1324} shows the case $h=(3,3,4,4)$ and $w=1324$, where both $\Gamma_{w,h}$ and $\Gamma_{\widetilde{w},h}$ are isomorphic to the graph of the hexagonal prism.

\begin{figure}
	\begin{center}
    \begin{subfigure}[b]{0.45\textwidth}
    \centering
	\begin{tikzpicture}[scale=.4]
		\coordinate (0) at (0,0);
		\foreach \i/\j in {-3/1,0/2,3/3}{
		\coordinate (\j) at (\i,2);}
		\foreach \i/\j in {-6/4,-3/5,0/6,3/7,6/8}{
		\coordinate (\j) at (\i,4);}
		\foreach \i/\j in {-7.5/9,-4.5/10,-1.5/11,1.5/12,4.5/13,7.5/14}{
		\coordinate (\j) at (\i,6);}
		\foreach \i/\j in {-6/15,-3/16,0/17,3/18,6/19}{
		\coordinate (\j) at (\i,8);}
		\foreach \i/\j in {-3/20,0/21,3/22}{
		\coordinate (\j) at (\i,10);}
		\coordinate (23) at (0,12);
		\draw[color=gray] (0) -- (1);
		\draw[color=gray] (0) -- (2);
		\draw[color=gray] (0) -- (3);
		\draw[color=gray] (0) -- (14);
		\draw[color=gray] (1) -- (4);
		\draw[color=gray] (1) -- (6);
		\draw[color=gray, bend left=10] (1) edge (16);
		\draw[color=gray] (2) -- (5);
		\draw[color=gray] (2) -- (7);
		\draw[color=gray] (2) -- (8);
		\draw[ultra thick, color=black] (4) -- (9);
		\draw[ultra thick, color=black] (4) -- (10);
		\draw[ultra thick, color=black] (4) -- (11);
		\draw[color=gray] (5) -- (9);
		\draw[color=gray] (5) -- (12);
		\draw[color=gray, bend right=10] (5) edge (20);
		\draw[ultra thick, color=black] (9) -- (15);
		\draw[ultra thick, color=black] (9) -- (17);
		\draw[color=gray] (3) -- (6);
		\draw[color=gray] (3) -- (7);
		\draw[color=gray] (3) -- (8);
		\draw[color=gray] (6) -- (10);
		\draw[color=gray] (6) -- (11);
		\draw[color=gray] (7) -- (12);
		\draw[color=gray] (7) -- (14);
		\draw[color=gray] (8) -- (13);
		\draw[color=gray] (8) -- (14);
		\draw[ultra thick, color=black] (10) -- (15);
		\draw[ultra thick, color=black] (10) -- (16);
		\draw[ultra thick, color=black] (11) -- (16);
		\draw[ultra thick, color=black] (11) -- (18);
		\draw[color=gray] (12) -- (15);
		\draw[color=gray] (12) -- (17);
		\draw[color=gray] (13) -- (18);
		\draw[color=gray] (13) -- (19);
		\draw[color=gray, bend right=10] (13) edge (23);
		\draw[color=gray] (14) -- (19);
		\draw[ultra thick, color=black] (15) -- (20);
		\draw[ultra thick, color=black] (16) -- (21);
		\draw[ultra thick, color=black] (17) -- (20);
		\draw[ultra thick, color=black] (17) -- (22);
		\draw[ultra thick, color=black] (18) -- (21);
		\draw[ultra thick, color=black] (18) -- (22);
		\draw[color=gray] (19) -- (21);
		\draw[color=gray] (19) -- (22);
		\draw[ultra thick, color=black] (20) -- (23);
		\draw[ultra thick, color=black] (21) -- (23);
		\draw[ultra thick, color=black] (22) -- (23);
		\foreach \i/\j in {0/1234, 1/1243, 2/1324, 3/2134, 4/1423, 5/1342, 6/2143, 7/3124, 8/2314}{
		\node[below] at (\i) {\scriptsize \j};} 
		\foreach \i/\j in {9/1432, 10/4123, 11/2413, 12/3142, 13/2341, 14/3214}{
		\node[below] at (\i) {\scriptsize \j};} 
		\foreach \i/\j in {15/4132, 16/4213, 17/3412, 18/2431, 19/3241, 20/4312, 21/4231, 22/3421, 23/4321}{
		\node[above] at (\i) {\scriptsize \j};} 
		\foreach \i in {0,1,...,23}{
		\filldraw[color=purple] (\i) circle (4pt);}
		\foreach \i in {0,1,2,3,5,6,7,8,12,13,14,19}{
		\filldraw[color=gray] (\i) circle (4pt);}
	\end{tikzpicture}
    \end{subfigure}
    \begin{subfigure}[b]{0.45\textwidth}
    \centering
	\begin{tikzpicture}[scale=.4]
		\coordinate (0) at (0,0);
		\foreach \i/\j in {-3/1,0/2,3/3}{
		\coordinate (\j) at (\i,2);}
		\foreach \i/\j in {-6/4,-3/5,0/6,3/7,6/8}{
		\coordinate (\j) at (\i,4);}
		\foreach \i/\j in {-7.5/9,-4.5/10,-1.5/11,1.5/12,4.5/13,7.5/14}{
		\coordinate (\j) at (\i,6);}
		\foreach \i/\j in {-6/15,-3/16,0/17,3/18,6/19}{
		\coordinate (\j) at (\i,8);}
		\foreach \i/\j in {-3/20,0/21,3/22}{
		\coordinate (\j) at (\i,10);}
		\coordinate (23) at (0,12);
		\draw[color=gray] (0) -- (1);
		\draw[color=gray] (0) -- (2);
		\draw[color=gray] (0) -- (3);
		\draw[color=gray] (0) -- (14);
		\draw[color=gray] (1) -- (4);
		\draw[color=gray] (1) -- (6);
		\draw[color=gray, bend left=10] (1) edge (16);
		\draw[ultra thick, color=black] (2) -- (5);
		\draw[ultra thick, color=black] (2) -- (7);
		\draw[ultra thick, color=black] (2) -- (8);
		\draw[color=gray] (4) -- (9);
		\draw[color=gray] (4) -- (10);
		\draw[color=gray] (4) -- (11);
		\draw[color=gray] (5) -- (9);
		\draw[ultra thick, color=black] (5) -- (12);
		\draw[ultra thick, color=black, bend right=10] (5) edge (20);
		\draw[color=gray] (9) -- (15);
		\draw[color=gray] (9) -- (17);
		\draw[color=gray] (3) -- (6);
		\draw[color=gray] (3) -- (7);
		\draw[color=gray] (3) -- (8);
		\draw[color=gray] (6) -- (10);
		\draw[color=gray] (6) -- (11);
		\draw[ultra thick, color=black] (7) -- (12);
		\draw[ultra thick, color=black] (7) -- (14);
		\draw[ultra thick, color=black] (8) -- (13);
		\draw[ultra thick, color=black] (8) -- (14);
		\draw[color=gray] (10) -- (15);
		\draw[color=gray] (10) -- (16);
		\draw[color=gray] (11) -- (16);
		\draw[color=gray] (11) -- (18);
		\draw[color=gray] (12) -- (15);
		\draw[ultra thick, color=black] (12) -- (17);
		\draw[color=gray] (13) -- (18);
		\draw[ultra thick, color=black] (13) -- (19);
		\draw[ultra thick, color=black, bend right=10] (13) edge (23);
		\draw[ultra thick, color=black] (14) -- (19);
		\draw[color=gray] (15) -- (20);
		\draw[color=gray] (16) -- (21);
		\draw[ultra thick, color=black] (17) -- (20);
		\draw[ultra thick, color=black] (17) -- (22);
		\draw[color=gray] (18) -- (21);
		\draw[color=gray] (18) -- (22);
		\draw[color=gray] (19) -- (21);
		\draw[ultra thick, color=black] (19) -- (22);
		\draw[ultra thick, color=black] (20) -- (23);
		\draw[color=gray] (21) -- (23);
		\draw[ultra thick, color=black] (22) -- (23);
		\foreach \i/\j in {0/1234, 1/1243, 2/1324, 3/2134, 4/1423, 5/1342, 6/2143, 7/3124, 8/2314}{
		\node[below] at (\i) {\scriptsize \j};} 
		\foreach \i/\j in {9/1432, 10/4123, 11/2413, 12/3142, 13/2341, 14/3214}{
		\node[below] at (\i) {\scriptsize \j};} 
		\foreach \i/\j in {15/4132, 16/4213, 17/3412, 18/2431, 19/3241, 20/4312, 21/4231, 22/3421, 23/4321}{
		\node[above] at (\i) {\scriptsize \j};} 
		\foreach \i in {0,1,...,23}{
		\filldraw[color=purple] (\i) circle (4pt);}
		\foreach \i in {0,1,3,4,6,9,10,11,15,16,18,21}{
		\filldraw[color=gray] (\i) circle (4pt);}
	\end{tikzpicture}
    \end{subfigure}
	\end{center}
	\caption{The induced subgraphs $\Gamma_{\widetilde{w},h}$ and $\Gamma_{w,h}$ of $\Gamma_{h}$ for a Hessenberg function $h=(3,3,4,4)$ and $w=1324$.}
	\label{fig:1324}
\end{figure}

\begin{proposition}\label{prop:Hessenberg Schubert intersection} 
If $\Omega_w\cap \Hess(S,h)$ is smooth and connected, then $w\in\gen$ and $\Gamma_{v,h}$ is a regular graph for any permutation $v$ with $\widetilde{v}=w$.
\end{proposition}
\begin{proof}
 Suppose that \(w \notin \mathcal{G}_h\).
Then \(\Omega_{w} \cap \Hess(S,h)\) is reducible by Lemma~\ref{lemm:gen_irr}.
Since we assume that it is connected, the intersection \(\Omega_{w} \cap \Hess(S,h)\) cannot be smooth.
This contradicts our assumption.
    
    For the latter part, by Lemma~\ref{lemm:iso}, it suffices to show the regularity of $\Gamma_{w,h}$ for $w\in\gen$. 
    
    Since $\Omega_w\cap\Hess(S,h)$ is a $T$-stable subvariety of $\mathcal{F}\ell(\C^n)$, if it is smooth, it follows from \cite[Corollary~2]{Brion2002} that the number of $T$-stable curves meeting at each $T$-fixed point equals the dimension of $\Owo{w}\cap\Hess(S,h)$. Therefore, $\Gamma_{w,h}$ is a regular graph.
\end{proof}

\section{Properties of  \texorpdfstring{$h$}{h}-Bruhat order and the graphs \texorpdfstring{$\Gamma_{w,h}$}{Gwh}}\label{sec:GKM}

In this section, we first introduce the notion of the $h$-Bruhat order and investigate the properties of the order in relation to the Bruhat order. Subsequently, in the second subsection, we proceed to look at the graph $\Gamma_{w,h}$ for $w\in\gen$; we define a nice injective map from the set of incident edges of $u$ to that of $v$ when $v$ is greater than $u$ in the $h$-Bruhat order. Finally, we use the injection to introduce seven special patterns that a permutation $w\in\gen$ has to avoid for $\Gamma_{w,h}$ to be regular.
\subsection{The $h$-Bruhat order}\label{sec:interval}

Proposition~\ref{prop:h-fixed points} characterizes the $T$-fixed points of Hessenberg Schubert varieties for $w\in\gen$ in relation to the Bruhat order; $\Ow{w,h}^T=[w,w_0]$. It is natural to ask if $\Ow{w,h}^T$ can be characterized in terms of an order on $\mathfrak S_n$, which respects the Hessenberg function $h$ in some way. Indeed, we can think of a natural order that respects $h$: 

\begin{definition}
Let $h\colon [n] \to [n]$ be a Hessenberg function.
 \begin{enumerate}
 \item The \emph{$h$-Bruhat order} on $\mathfrak S_n$ is the transitive closure of the following relation
 $$u\prec_h v \quad \text{ if and only if }\quad v=u(i,j) \text{ for some $(i,  j)$ such that $j\leq h(i)$ and } \ell(v)>\ell(u)\,.$$
 \item For permutations $u\prec_h v$, let $[u, v]_h$ denote the \emph{$h$-Bruhat interval} 
 $\{w \mid u\preceq_h w \preceq_h v\}$. 
 \end{enumerate}
\end{definition}

By definition, the Bruhat order refines the $h$-Bruhat order; that is, if $u\prec_h v$, then $u\prec v$.
The Hasse diagram of the $h$-Bruhat order is connected if and only if $h(i)>i$ for all $i<n$, which, by Remark~\ref{remark:Hessenberg function} (3), is equivalent to the connectedness of $\Hess(S, h)$.
When $h(i)=i$ for some $i<n$, the unique maximum permutation $w_0$ in the Bruhat order is not the maximum in the $h$-Bruhat order: if $h=(2,3,3,5,5)$, then $54231\not\prec_h 54321=w_0$. However, we obtain the following lemma when $w\in\gen$.

\begin{lemma}\label{lem:h-maximum} For a Hessenberg function $h$ and $w\in\gen$, if $w\preceq_h u$, then $u \preceq_h w_0$.
\end{lemma}

\begin{proof}
Let $h$ be a Hessenberg function on $[n]$.
If $h(i)>i$ for all $i<n$, then we can successively apply the adjacent transpositions to $u$ on the right to make a chain $u\prec_h u_1 \prec_h u_2\prec_h\cdots\prec_h u_l\prec_h w_0$; first move $1$ to the $n$th position, then move $2$ to the $(n-1)$st position, and so on until we obtain $w_0$. 

If $h(i)=i$ for some $i<n$, let 
$0=i_0<i_1<i_2<\cdots <i_{k}=n$ be the indices satisfying $h(i_j)=i_j$ for $j \in [k]$, and 
set $I_j:=[i_{j-1}+1, i_j]$.
Then $[n]=\bigsqcup_{j}I_j$. We first observe that if $x=w(i)$ for $i\in I_j$ then $x+1=w(i')$ for $i'\in I_{j'}$ with $j'\leq j$ since $w\in\gen$. This implies that $w(I_j)= \{w(i_{j-1}), \dots,  w(i_j)\}=[n-i_j+1,n-i_{j-1}]$ for all $j\in [k]$. We then consider the condition $w\preceq_h u$. By the definition of the $h$-Bruhat order, to obtain $u$ from $w$ we can only apply the transpositions inside each $I_j$; that is $(i, i')$'s such that $i, i'\in I_j$ for some $j$. This shows that $u(I_j)=w(I_j)$ for all $j\in [k]$. For each $j$, we can consider the restriction $u^j:=u|_{I_j}$ of $u$ on $I_j$ as a permutation on $[n-i_j+1,n-i_{j-1}]$ and we successively apply adjacent transpositions inside each $I_j$ to obtain the permutation $w_0^j:=(n-i_{j-1})\, (n-i_{j-1}-1)\, \cdots\, (n-i_j+1)$ as we did in the proof of the first case that $h(i)>i$. This defines a chain $u^j\prec_h u^j_1\prec_h\cdots \prec_h u^j_{l(j)}=w_0^j$, and we let $t^j_i$ be the transposition satisfying $u^j_{i-1}t^j_i= u^j_i$ for $i\in [l(j)]$.
We can now construct a chain
\(u\prec_h u_1\prec_h \cdots  \prec_h u_{l}=w_0  \) where $l=\prod_j l(j)$, by applying the transpositions $t^j_i$ in the following order: $t^1_1, \dots, t^1_{l(1)}, t^2_1, \dots, t^2_{l(2)}, \dots, t^k_1, \dots, t^k_{l(k)}$. This completes the proof.
\end{proof}

There is a nice result on the $T$-fixed set of $\Omega_{w,h}$ in relation to the $h$-Bruhat order.
 
\begin{proposition}\cite[Proposition 2.11]{CHL}\label{prop:h-interval}
Let $h$ be a Hessenberg function and $w$ be a permutation. If $v\in \Ow{w,h}^T$, then $w\preceq_h v$. In particular, for $w\in\gen$, if $w\preceq v$, then $w\preceq_h v$.
\end{proposition}

This proposition implies that for $w\in\gen$, the graph $\Gamma_{w,h}$ is connected, and equivalently, the intersection $\Ow{w}\cap\Hess(S,h)$ is connected.

Now, we can characterize the $T$-fixed points of the Hessenberg Schubert variety in terms of the $h$-Bruhat order. 
Proposition~\ref{prop:h-interval} together with Lemma~\ref{lem:h-maximum} prove the following theorem.

\begin{theorem}\label{thm:interval}
Let $w$ be a permutation in $\gen$ for a Hessenberg function $h$. Then, as sets
 $$[w, w_0]_h=[w, w_0]\,.$$ 
\end{theorem}

The graph $\Gamma_h$ in Figure~\ref{fig:2134} is the Hasse diagram of the $h$-Bruhat order when $h=(3, 3, 4, 4)$. One can check that $w=2134\in \gen$ and $[w, w_0]_h=[w, w_0]$, while $w'=4213 \not \in \gen$ and $[w', w_0]_h=\{4213, 4231, 4321 \} \neq [w', w_0]=[w', w_0]_h\cup \{ 4312  \}$.

\begin{figure}
	\begin{center}
	\begin{tikzpicture}[scale=.6]
		\coordinate (0) at (0,0);
		\foreach \i/\j in {-3/1,0/2,3/3}{
		\coordinate (\j) at (\i,2);}
		\foreach \i/\j in {-6/4,-3/5,0/6,3/7,6/8}{
		\coordinate (\j) at (\i,4);}
		\foreach \i/\j in {-7.5/9,-4.5/10,-1.5/11,1.5/12,4.5/13,7.5/14}{
		\coordinate (\j) at (\i,6);}
		\foreach \i/\j in {-6/15,-3/16,0/17,3/18,6/19}{
		\coordinate (\j) at (\i,8);}
		\foreach \i/\j in {-3/20,0/21,3/22}{
		\coordinate (\j) at (\i,10);}
		\coordinate (23) at (0,12);
		\draw[color=gray] (0) -- (1);
		\draw[color=gray] (0) -- (2);
		\draw[color=gray] (0) -- (3);
		\draw[color=gray] (0) -- (14);
		\draw[color=gray] (1) -- (4);
		\draw[color=gray] (1) -- (6);
		\draw[color=gray, bend left=10] (1) edge (16);
		\draw[color=gray] (2) -- (5);
		\draw[color=gray] (2) -- (7);
		\draw[color=gray] (2) -- (8);
		\draw[color=gray] (4) -- (9);
		\draw[color=gray] (4) -- (10);
		\draw[color=gray] (4) -- (11);
		\draw[color=gray] (5) -- (9);
		\draw[color=gray] (5) -- (12);
		\draw[color=gray, bend right=10] (5) edge (20);
		\draw[color=gray] (9) -- (15);
		\draw[color=gray] (9) -- (17);
		\draw[ultra thick, color=black] (3) -- (6);
		\draw[ultra thick, color=black] (3) -- (7);
		\draw[ultra thick, color=black] (3) -- (8);
		\draw[ultra thick, color=black] (6) -- (10);
		\draw[ultra thick, color=black] (6) -- (11);
		\draw[ultra thick, color=black] (7) -- (12);
		\draw[ultra thick, color=black] (7) -- (14);
		\draw[ultra thick, color=black] (8) -- (13);
		\draw[ultra thick, color=black] (8) -- (14);
		\draw[ultra thick, color=black] (10) -- (15);
		\draw[ultra thick, color=black] (10) -- (16);
		\draw[ultra thick, color=black] (11) -- (16);
		\draw[ultra thick, color=black] (11) -- (18);
		\draw[ultra thick, color=black] (12) -- (15);
		\draw[ultra thick, color=black] (12) -- (17);
		\draw[ultra thick, color=black] (13) -- (18);
		\draw[ultra thick, color=black] (13) -- (19);
		\draw[ultra thick, color=black, bend right=10] (13) edge (23);
		\draw[ultra thick, color=black] (14) -- (19);
		\draw[ultra thick, color=black] (15) -- (20);
		\draw[ultra thick, color=black] (16) -- (21);
		\draw[ultra thick, color=black] (17) -- (20);
		\draw[ultra thick, color=black] (17) -- (22);
		\draw[ultra thick, color=black] (18) -- (21);
		\draw[ultra thick, color=black] (18) -- (22);
		\draw[ultra thick, color=black] (19) -- (21);
		\draw[ultra thick, color=black] (19) -- (22);
		\draw[ultra thick, color=black] (20) -- (23);
		\draw[ultra thick, color=black] (21) -- (23);
		\draw[ultra thick, color=black] (22) -- (23);
		\foreach \i/\j in {0/1234, 1/1243, 2/1324, 3/2134, 4/1423, 5/1342, 6/2143, 7/3124, 8/2314}{
		\node[below] at (\i) {\scriptsize \j};} 
		\foreach \i/\j in {9/1432, 10/4123, 11/2413, 12/3142, 13/2341, 14/3214}{
		\node[below] at (\i) {\scriptsize \j};} 
		\foreach \i/\j in {15/4132, 16/4213, 17/3412, 18/2431, 19/3241, 20/4312, 21/4231, 22/3421, 23/4321}{
		\node[above] at (\i) {\scriptsize \j};} 
		\foreach \i in {0,1,...,23}{
		\filldraw[color=purple] (\i) circle (4pt);}
		\foreach \i in {0,1,2,4,5,9}{
		\filldraw[color=gray] (\i) circle (4pt);}
		\foreach \i in {13,18,19,21,22,23}{
		\filldraw[color=black] (\i) circle (4pt);}
	\end{tikzpicture}
	\end{center}
	\caption{The induced subgraph $\Gamma_{w,h}$ of $\Gamma_{h}$ for a Hessenberg function $h=(3,3,4,4)$ and $w=2134\in\gen$.}
	\label{fig:2134}
\end{figure}

Finally, we show that the chain property of the Bruhat order extends to the $h$-Bruhat order.

\begin{theorem}[Chain Property of the $h$-Bruhat order]\label{thm:h-chain}
    Let $h$ be a Hessenberg function on $[n]$. If $u\prec_h v$ in $\mathfrak S_n$, then there exist $v_i\in \mathfrak S_n$ satisfying $\ell(v_i)=\ell(u)+i$ for $0\leq i\leq k$, and $u=v_0\prec_h v_1\prec_h \cdots \prec_h v_k=v$. 
\end{theorem}

\begin{proof} We show that `if  $u\prec_h u(a,b)$ with $b\leq h(a)$ then there is a chain $u=v_0\prec_h v_1 \prec_h\cdots\prec_h v_k\coloneqq u(a, b)$ with $\ell(v_i)=\ell(u)+i$ for $i=1, \dots, k$.' This will prove the theorem by the definition of the $h$-Bruhat order. 

Suppose that $u\prec_h u(a,b)$ with $b\leq h(a)$, then by Proposition~\ref{prop:chain} there is a chain 
  $u=v_0\prec v_1\prec \cdots \prec v_k=u(a,b)$ with $\ell(v_i)=\ell(u)+i$ for $0\leq i\leq k$.
The main observation is that, if we only look at the subsequence $u(a) u(a+1) \cdots u(b)$ of $u$ and $u(b) u(a+1) \cdots u(a)$ of $u(a,b)$, then they are permutations of $\{u(a), u(a+1) \cdots, u(b)\}$. Thus, we can locally apply the chain property to obtain a chain from $u(a) u(a+1) \cdots u(b)$ to $u(b) u(a+1) \cdots u(a)$, which means that we can assume that $v_{i+1}=v_i(a_i, b_i)$ for $a\leq a_i< b_i\leq b$. This implies that  $b_i\leq h(a_i)$ and $v_i\prec_h v_{i+1}$, which completes the proof.
\end{proof}

\subsection{An injection for the GKM graphs \texorpdfstring{$\Gamma_{w,h}$}{Gwh}}\label{sec:injection}

Recall that 
for a Hessenberg function $h$, if $w\in\gen$, then $\Gamma_{w,h}$ is the GKM graph of $\Ow{w}\cap\Hess(S,h)$ and it is the induced subgraph of $\Gamma_h$ whose vertex set is $[w, w_0]=[w, w_0]_h$. Note that $\{u,v\}$ is an edge of $\Gamma_{w,h}$ if and only if $w\preceq u,v$ and $u=v(i,j)$ for some $(i,j)$ with $j \leq h(i)$. For each vertex $u$ in $\Gamma_{w,h}$, we denote by
$$
E_{w,h}(u) \coloneqq \{ (i,j) \mid u(i,j) \succeq w \text{~for~} 1\leq i<j\leq h(i) \}
$$
the set of transpositions $(i,j)$ satisfying that $\{u,u(i,j)\}$ is an edge of $\Gamma_{w,h}$. Thus, it is evident that 
$$
E_{w,h}(w)=\{ (i,j) \mid w(i)<w(j) \text{~for~} 1\leq i<j\leq h(i) \}\,.
$$

For vertices $u$ and $v$ in $\Gamma_{w,h}$ satisfying that $v=u(a,b)$ for some $(a,b) \in E_{w,h}(u)$ and $u\prec_h v$, we define a map $\phi_{uv} \colon E_{w,h}(u) \to E_{w,h}(v)$ by $(i,j) \mapsto (\overline{i},\overline{j})$, where
$$
(\overline{i},\overline{j}) \coloneqq \begin{cases}
(b,j) & \mbox{if $i=a$, $j>b$, and $(b,j)\not \in E_{w,h}(u) $},\\
(i,a) & \mbox{if $i<a$, $j=b$, and $(i,a)\not \in E_{w,h}(u) $},\\
(i,j) & \mbox{otherwise}.
\end{cases}
$$

The map $\phi_{uv}$ is defined for two vertices $u$ and $v$ that are connected by an edge in $\Gamma_{w,h}$. 
It is defined almost as the identity map. 
Since 
\[
\{(i,j)\in E_{w,h}(u)\mid \{i,j\}\cap\{a,b\}=\emptyset\}
=
\{(i,j)\in E_{w,h}(v)\mid \{i,j\}\cap\{a,b\}=\emptyset\},
\]
it is natural to define the map $\phi_{uv}$ as the identity on this subset. 
When $\{i,j\}\cap\{a,b\}\neq\emptyset$, the definition of $\phi_{uv}$ takes a different form; 
in particular, the first and the second cases of the definition correspond to specific subcases of this situation, 
where $\phi_{uv}$ differs from the identity map.
 We exhibit two examples of the map $\phi_{uv}$.

 \begin{example} 
Consider a permutation $w=2134\in\gen$ for the Hessenberg function $h=(3,3,4,4)$.
The graph $\Gamma_{w,h}$ is depicted in Figure~\ref{fig:2134}.

\begin{enumerate}
    \item If $u=2314$ and $v=u(3,4)=2341$, then we have $E_{w,h}(u)=\{(1,2),(2,3),(3,4)\}$ and $E_{w,h}(v)=\{(1,2), (1,3), (2,3),(3,4)\}$. The map $\phi_{uv}\colon E_{w,h}(u) \to E_{w,h}(v)$ is an injection defined by $\phi_{uv}(i,j)=(i,j)$ for all $(i,j)\in E_{w,h}(u)$.
    \item If $u=3124$ and $v=u(2,3)=3214$, then we have $E_{w,h}(u)=\{(1,3),(2,3),(3,4)\}$ and $E_{w,h}(v)=\{(1,2),(2,3),(3,4)\}$. The map $\phi_{uv}$ sends $(1,3)$ to $(1,2)$ and fixes the others.
\end{enumerate}

\end{example}

\begin{lemma}\label{lem:injection}
Let $w$ be a permutation in $\gen$ for a Hessenberg function $h$. If $u$ and $v$ are vertices in $\Gamma_{w,h}$ satisfying that $v=u(a,b)$ for some $(a,b) \in E_{w,h}(u)$ and $u\prec_h v$, then the map $\phi_{uv} \colon E_{w,h}(u) \to E_{w,h}(v)$ is well-defined and injective.
\end{lemma}
\begin{proof}
From the definition of $\phi_{uv}$, it is enough to show that the map $\phi_{uv}$ is well-defined. 
Indeed, for every $(i,j)$ with $\{i,j\}\cap\{a,b\}\neq\emptyset$, we have 
\[
\phi_{uv}((i,j)) \in \{(\bar{i},\bar{j})\in E_{w,h}(v)\mid \{\bar{i},\bar{j}\}\cap\{a,b\}\neq\emptyset\}.
\]
Since the images of distinct elements in $E_{w,h}(u)$ remain disjoint in $E_{w,h}(v)$, 
no two different pairs in the domain can be mapped to the same element, 
and thus $\phi_{uv}$ is injective. Specifically, suppose that \((\bar{i}, \bar{j}) = (b, j) = (\bar{k}, \bar{\ell}) = (k, \ell)\).
When \((\bar{i}, \bar{j}) = (b, j)\), either  
\begin{itemize}
    \item \(i = a\), \(j > b\), and \((b, j) \notin E_{w,h}(u)\), or
    \item \((i, j) = (b, j)\).
\end{itemize}
However, the above two cases can never occur simultaneously, and hence we must have \((i, j) = (k, \ell)\).

Note that from the assumption, we have  $a<b\leq h(a)$, $u(a)<u(b)$, and $v(a,b)=u \succeq w$.
Let $(i,j)\in E_{w,h}(u)$. We show that $\phi_{uv}$ is well-defined by showing $\overline{i}<\overline{j}\leq h(\overline{i})$ and $ v(\overline{i},\overline{j})\succeq w$.  If $(i,j)=(a,b)$, then $\phi_{uv}(a,b)=(a,b)$. If $\{i,j\} \cap \{a,b\}=\emptyset$, then $(\overline{i},\overline{j})=(i,j)$ so that we have $\overline{i}<\overline{j}\leq h(\overline{i})$ and
$$
v(\overline{i},\overline{j})=v(i,j)=u(a,b)(i,j)=u(i,j)(a,b)\succ u(i,j) \succeq w\,.
$$
Hence, if $|\{i,j\}\cap\{a,b\}| \neq 1$, then $\phi_{uv}$ is well-defined. For the remaining $(i,j)$, we split them into four cases as follows.

\textbf{Case i)} Let $i=a$ and $j < b$. Since $\phi_{uv}(a,j)=(a,j)$, we only need to show that $v(a,j)\succeq w$. In this case, since $(u(a,j))(j)=u(a)<u(b)=(u(a,j))(b)$, by Lemma~\ref{lemma:Bruhat order}
$$
v(a,j)=u(a,j)(j,b)\succ u(a,j) \succeq w\,.
$$

\textbf{Case ii)} Let $i>a$ and $j=b$. Similarly to the above case, since $\phi_{uv}(i,b)=(i,b)$, let us show $v(i,b)\succeq w$. Since $(u(i,b))(a)=u(a)<u(b)=(u(i,b))(i)$, we have
$$
v(i,b)=u(i,b)(a,i)\succ u(i,b) \succeq w\,.
$$

\textbf{Case iii)} Let $i<a$ and $j\in\{a,b\}$. 
First, consider the subcase $j=a$ so that $(i,a)\in E_{w,h}(u)$ and $\phi_{uv}(i,a)=(i,a)$.
Then since $(u(i,a))(i)=u(a)<u(b)=(u(i,a))(b)$, we have
$$
v(i,a)=u(a,b)(i,a)=u(i,a)(i,b)\succ u(i,a) \succeq w\,,
$$ 
and $\phi_{uv}(i,a)=(i,a)$ is in $E_{w,h}(v)$.

Now consider the subcase $j=b$ so that $(i,b)\in E_{w,h}(u)$. 
If $(i,a)\in E_{w,h}(u)$, then $\phi_{uv}(i,b)=(i,b)$.
Then since $u(i,a),u(i,b)\succeq w$, and 
$$
(v(i,b))[k]=\begin{cases}
(u(i,a))[k] & \mbox{if $k<a$},\\
(u(i,b))[k] & \mbox{otherwise},
\end{cases}
$$
for all $k$, we have $v(i,b)\succeq w$ by Proposition~\ref{prop:Bruhat order} (2) and $\phi_{uv}(i,b)=(i,b)$ is  in $E_{w,h}(v)$.

On the other hand, if $(i,a)\not\in E_{w,h}(u)$, 
then $\phi_{uv}(i,b)=(i,a)$. Here, $i<a<b\leq h(i)$. In this case, we need to show that $v(i,a)\succeq w$. If $u(i)<u(b)$, then $v(i,a)\succ v \succeq w$. Otherwise, we have $u(a)<u(b)<u(i)$, so it follows that 
$$
v(i,a)=u(i,b)(a,b)\succ u(i,b)\succeq w
$$
since $(u(i,b))(a)=u(a)<u(i)=(u(i,b))(b)$.

\textbf{Case iv)} Let $i\in\{a,b\}$ and $j>b$. 
We proceed as in Case (iii). 
First, consider the subcase $i=b$ so that $(b,j)\in E_{w,h}(u)$ and $\phi_{uv}(b,j)=(b,j)$.
Since $(u(b,j))(a)=u(a)<u(b)=(u(b,j))(j)$, we have
$$
v(b,j)=u(a,b)(b,j)=u(b,j)(a,j)\succ u(b,j) \succeq w\,.
$$
Now consider the subcase $i=a$ so that $(a,j)\in E_{w,h}(u)$. 
If $(b,j)\in E_{w,h}(u)$, then $\phi_{uv}(a,j)=(a,j)$. Then
$v(a,j)\succeq w$ as $u(a,j),u(b,j)\succeq w$ and for all $k$
\[
(v(a,j))[k]=\begin{cases}
(u(a,j))[k] & \mbox{if 
$k<b$,
}\\
(u(b,j))[k] & \mbox{otherwise}\,,
\end{cases}
\] 
and $\phi_{uv}(a,j)=(a,j)$ is  in $E_{w,h}(v)$.

On the other hand, if $(b,j)\not\in E_{w,h}(u)$,
then $\phi_{uv}(a,j)=(b,j)$. Here, $a<b<j\leq h(a)\leq h(b)$. In this case, we need to show that $v(b,j)\succeq w$. If $u(a)<u(j)$, then $v(b,j)\succ v \succeq w$. Otherwise, we have $u(j)<u(a)<u(b)$, so it follows that 
$$
v(b,j)=u(a,j)(a,b)\succ u(a,j)\succeq w
$$
since $(u(a,j))(a)=u(j)<u(b)=(u(a,j))(b)$.

Thus, we conclude that $\phi_{uv}$ is well-defined.
\end{proof}

The following theorem is immediate from the previous lemma due to Lemma~\ref{lem:h-maximum} and Theorem~\ref{thm:h-chain}.

\begin{theorem}\label{thm:increasing}
For a Hessenberg function $h$ and a permutation {$w\in\gen$}, if $u$ and $v$ are vertices in $\Gamma_{w,h}$ with $u\preceq_h v$, then $|E_{w,h}(u)|\leq |E_{w,h}(v)|$. Furthermore, $|E_{w,h}(w)|\leq |E_{w,h}(w_0)|.$
\end{theorem}

For example, Figure~\ref{fig:2134} shows that strict inequality may hold in the previous theorem.
Here, $|E_{w,h}(u)|=3$ and $|E_{w,h}(v)|=4$ for red vertices $u$ and black vertices $v$, respectively.

\subsection{Associated patterns for irregularity}\label{sec:irregularity}

Recall that for a Hessenberg function $h=(n,n,\dots,n)$, the induced subgraph $\Gamma_{w,h}$ is irregular if and only if $w\in \mathfrak{S}_n$ contains pattern $2143$ or pattern $1324$. We provide seven patterns associated with a Hessenberg function $h$ for irregularity of $\Gamma_{w,h}$ when $w\in\gen$. Using the injection $\phi_{uv}$ that we defined in the previous section, we show that $\Gamma_{w,h}$ is not regular when $w\in\gen$ contains one of the seven patterns.

\begin{definition}\label{def:pattern4}

Let $h$ be a Hessenberg function on $[n]$. For a permutation $w\in \mathfrak{S}_n$, we say that $w$ contains the associated pattern
\begin{enumerate}
\item $\hpat{2143}$ if $w(j)<w(i)<w(\ell)<w(k)$ for some $i<j<k<\ell\leq h(i)$;
\item $\hpat{1324}$ if $w(i)<w(k)<w(j)<w(\ell)$ for some $i<j<k<\ell\leq h(j)$ with $k\leq h(i)$;
\item $\hpat{1243}$ if $w(i)<w(j)<w(\ell)<w(k)$ for some $i<j<k<\ell\leq h(j)$ with $j\leq h(i)<\ell$;
\item $\hpat{2134}$ if $w(j)<w(i)<w(k)<w(\ell)$ for some $i<j<k<\ell\leq h(k)$ with $k\leq h(i)<\ell$; 
\item $\hpat{1423}$ if 
$w(i)<w(k)<w(\ell)<w(j)$ 
for some $i<j<k<\ell\leq h(j)$ with $k\leq h(i)<\ell$;
\item $\hpat{2314}$ if 
$w(k)<w(i)<w(j)<w(\ell)$ 
for some $i<j<k<\ell\leq h(j)$ with $k\leq h(i)<\ell$;
\item $\hpat{2413}$ if 
$w(k)<w(i)<w(\ell)<w(j)$ 
for some $i<j\leq h(i)<k\leq h(j)<\ell\leq h(k)$.\end{enumerate}
\end{definition}

Figure~\ref{fig:patterns} presents \emph{all} possible induced subgraphs with the set $\{j, j, k, \ell\}$ of vertices of the incomparability graph of the  Hessenberg function $h$, for each associated pattern defined in Definition~\ref{def:pattern4}. 

\begin{figure}[ht]
	\begin{center}
	\begin{tikzpicture}[scale=0.55]
		\draw (-5,2.4) -- (25,2.4) -- (25,-4.3) -- (-5,-4.3) -- (-5,2.4);
		\draw (-5,1.2) -- (25,1.2);
		\foreach \i in {0,5,10,15,20}{
		\draw (\i,-4.3) -- (\i,2.4);}
		\node at (-2.5,1.75) {$\hpat{2143}$};
		\foreach \i/\j in {1/i,2/j,3/k,4/\ell}{
		\coordinate (\i) at (-5+\i,0);
		\filldraw[color=black] (\i) circle (2pt);
		\node[below] at (\i) {$\j$};}
		\draw[thick] (1) -- (4);
		\draw[thick, bend left=80] (1) edge (3);
		\draw[thick, bend left=80] (2) edge (4);
		\draw[thick, bend left=80] (1) edge (4);
		\node at (2.5,1.75) {$\hpat{1324}$};
		\foreach \i/\j in {1/i,2/j,3/k,4/\ell}{
		\coordinate (\i) at (\i,0);
		\filldraw[color=black] (\i) circle (2pt);
		\node[below] at (\i) {$\j$};}
		\draw[thick] (1) -- (4);
		\draw[thick, bend left=80] (1) edge (3);
		\draw[thick, bend left=80] (2) edge (4);
		
		\foreach \i/\j in {1/i,2/j,3/k,4/\ell}{
		\coordinate (\i) at (\i,-3);
		\filldraw[color=black] (\i) circle (2pt);
		\node[below] at (\i) {$\j$};}
		\draw[thick] (1) -- (4);
		\draw[thick, bend left=80] (1) edge (3);
		\draw[thick, bend left=80] (2) edge (4);
		\draw[thick, bend left=80] (1) edge (4);
		\node at (7.5,1.78) {$\hpat{1243}$};
		\foreach \i/\j in {1/i,2/j,3/k,4/\ell}{
		\coordinate (\i) at (5+\i,0);
		\filldraw[color=black] (\i) circle (2pt);
		\node[below] at (\i) {$\j$};}
		\draw[thick] (1) -- (4);
		\draw[thick, bend left=80] (2) edge (4);
		
		\foreach \i/\j in {1/i,2/j,3/k,4/\ell}{
		\coordinate (\i) at (5+\i,-3);
		\filldraw[color=black] (\i) circle (2pt);
		\node[below] at (\i) {$\j$};}
		\draw[thick] (1) -- (4);
		\draw[thick, bend left=80] (1) edge (3);
		\draw[thick, bend left=80] (2) edge (4);
		\node at (12.5,1.78) {$\hpat{2134}$};
		\foreach \i/\j in {1/i,2/j,3/k,4/\ell}{
		\coordinate (\i) at (10+\i,0);
		\filldraw[color=black] (\i) circle (2pt);
		\node[below] at (\i) {$\j$};}
		\draw[thick] (1) -- (4);
		\draw[thick, bend left=80] (1) edge (3);
		
		\foreach \i/\j in {1/i,2/j,3/k,4/\ell}{
		\coordinate (\i) at (10+\i,-3);
		\filldraw[color=black] (\i) circle (2pt);
		\node[below] at (\i) {$\j$};}
		\draw[thick] (1) -- (4);
		\draw[thick, bend left=80] (1) edge (3);
		\draw[thick, bend left=80] (2) edge (4);
		\node at (17.5,1.75) {$\hpat{1423},  ~\hpat{2314}$};
		\foreach \i/\j in {1/i,2/j,3/k,4/\ell}{
		\coordinate (\i) at (15+\i,0);
		\filldraw[color=black] (\i) circle (2pt);
		\node[below] at (\i) {$\j$};}
		\draw[thick] (1) -- (4);
		\draw[thick, bend left=80] (1) edge (3);
		\draw[thick, bend left=80] (2) edge (4);
		\node at (22.5,1.78) {$\hpat{2413}$};
		\foreach \i/\j in {1/i,2/j,3/k,4/\ell}{
		\coordinate (\i) at (20+\i,0);
		\filldraw[color=black] (\i) circle (2pt);
		\node[below] at (\i) {$\j$};}
		\draw[thick] (1) -- (4);
	\end{tikzpicture}
	\end{center}
	\caption{Seven associated patterns with induced subgraphs of $\mathrm{inc}(h)$ on $\{i,j,k,\ell\}$.}
	\label{fig:patterns}
\end{figure}

Since all the patterns in Definition~\ref{def:pattern4} have length four, every incomparability graph whose connected components consist of at most three vertices automatically avoids all such patterns.
Therefore, to detect the occurrence of an associated pattern, it suffices to examine only the parts of the graph that contain connected components with at least four vertices. 
To check whether \(w\in \gen\) contains the associated pattern \(\hpat{p}\),
we first determine whether \(w\) itself contains the permutation pattern \(p\).
If it does, i.e., $w(i)w(j)w(k)w(\ell)$ has the same relative order as $p$, then we examine the induced subgraph of \(\mathrm{inc}(h)\) obtained by restricting to the vertices \(\{i,j,k,\ell\}\),
and verify whether it coincides with the corresponding graph shown in Figure~\ref{fig:patterns}.

\begin{example}\label{ex:pattern} 
Let $h=(3, 5, 6, 6, 6, 7, 8, 8)$ be a Hessenberg function. Its incomparability graph is given in Figure~\ref{fig:induced_inc}.
Consider the permutation $w=  2\,7 \,3\,1 \,8\,4\,5\,6\in\gen$.

For illustration, we examine several of the subwords of $w$.
To determine whether each subword corresponds to an associated pattern, 
we consider the induced subgraph of $\mathrm{inc}(h)$ corresponding to that subword.

\begin{itemize}
\item The subwords $3184$, $2738$, $2714$, and $1845$ correspond to the patterns
$2143$, $1324$, $2413$, and $1423$, respectively.
\item The subgraphs of $\mathrm{inc}(h)$ induced by 
$\{3,4,5,6\}$, $\{1,2,3,5\}$, $\{1,2,4,6\}$, and $\{4,5,6,7\}$ 
are shown in Figure~\ref{fig:induced_inc}.
\end{itemize}
Based on these, $w$ has the associated patterns 
$\hpat{2143}$, $\hpat{1324}$, and $\hpat{2413}$.
However, the subword $1845$ does not yield the pattern 
$\hpat{1423}$, since the induced subgraph on $\{4,5,6,7\}$ 
lacks the edge $\{5,7\}$.

\end{example}

\begin{figure}[ht]  
 \begin{tikzpicture}
        \foreach \i in {1,2,3,4, 5,6, 7, 8}{
		\coordinate (\i) at (-5+\i,0);
		\filldraw[color=black] (\i) circle (2pt);
		\node[below] at (\i) {$\i$};}
		\draw[thick] (1) -- (8);
		\draw[thick, bend left=80] (1) edge (3);
		\draw[thick, bend left=80] (2) edge (4);
		\draw[thick, bend left=80] (2) edge (5);
		\draw[thick, bend left=80] (3) edge (5);
		\draw[thick, bend left=80] (3) edge (6);
		\draw[thick, bend left=80] (4) edge (6);
		\node at (-0.4,-1) {$\mathrm{inc}(h)$};
    \end{tikzpicture}\\[3ex]

 \begin{tikzpicture}[scale=0.7]
 \foreach \i/\j in {1/i,2/j,3/k,4/\ell}{
		\coordinate (\i) at (0+\i,0);
		\filldraw[color=black] (\i) circle (2pt);}
		\node[below] at (1) {$3$};
		\node[below] at (2) {$4$};
		\node[below] at (3) {$5$};
		\node[below] at (4) {$6$};
		\draw[thick, bend left=80] (1) edge (3);
		\draw[thick, bend left=80] (1) edge (4);
		\draw[thick, bend left=80] (2) edge (4);
		\draw[thick] (1) -- (4);	
		
\foreach \i/\j in {1/i,2/j,3/k,4/\ell}{
		\coordinate (\i) at (5+\i,0);
		\filldraw[color=black] (\i) circle (2pt);}
		\node[below] at (1) {$1$};
		\node[below] at (2) {$2$};
		\node[below] at (3) {$3$};
		\node[below] at (4) {$5$};
		\draw[thick, bend left=80] (1) edge (3);
		\draw[thick, bend left=80] (2) edge (4);
		\draw[thick] (1) -- (4);

\foreach \i/\j in {1/i,2/j,3/k,4/\ell}{
		\coordinate (\i) at (15+\i,0);
		\filldraw[color=black] (\i) circle (2pt);}
		\node[below] at (1) {$4$};
		\node[below] at (2) {$5$};
		\node[below] at (3) {$6$};
		\node[below] at (4) {$7$};
		\draw[thick, bend left=80] (1) edge (3);
		\draw[thick] (1) -- (4);

		
\foreach \i/\j in {1/i,2/j,3/k,4/\ell}{
		\coordinate (\i) at (10+\i,0);
		\filldraw[color=black] (\i) circle (2pt);}
		\node[below] at (1) {$1$};
		\node[below] at (2) {$2$};
		\node[below] at (3) {$4$};
		\node[below] at (4) {$6$};
		\draw[thick] (1) -- (4);
	\end{tikzpicture}
\caption{The incomparability graph of $h=(3,5,6,6,6,7,8,8)$ and several of its induced subgraphs.}\label{fig:induced_inc}
\end{figure}

\begin{remark}\label{rmk:permuto}
\begin{enumerate}
\item When $h=(2, 3, \dots, n, n)$, the corresponding Hessenberg variety is the permutohedral variety, which is a smooth projective toric variety. From the definition, every permutation avoids the patterns (1) to (6). Furthermore, every permutation in $\gen$ of the permutohedral variety avoids the pattern ${2413}$, see \cite[Lemma~5.21]{CHL}. For $w\in\gen$, $\Gamma_{w,h}$ forms a face containing $w$ and $w_0$ in the permutohedron and the Hessenberg Schubert variety $\Owh{w}$ is the smooth toric variety corresponding to the face. Furthermore, for $w\in\Sn{n}$, the Hessenberg Schubert variety $\Owh{w}$ is isomorphic to $\Owh{\widetilde{w}}$ and it is the smooth toric variety corresponding to the face of the permutohedron whose $1$-skeleton forms $\Gamma_{w,h}$. See Figure~\ref{fig:2344}.
\item Recently, Insko, Precup, and Woo~\cite{IPW} classified all singular permutation flags in each regular Hessenberg variety with the Hessenberg function \((2,3,\dots,n,n)\) of type \(A\), relating them to certain combinatorial patterns, and further generalized these results to all Lie types using root-theoretic data.
Since we only consider regular semisimple Hessenberg varieties, we obtain the same result only in the case corresponding to (1).
\end{enumerate}
\end{remark}

\begin{figure}
	\begin{center}
	\begin{tikzpicture}[scale=.6]
		\coordinate (0) at (0,0);
		\foreach \i/\j in {-3/1,0/2,3/3}{
		\coordinate (\j) at (\i,2);}
		\foreach \i/\j in {-6/4,-3/5,0/6,3/7,6/8}{
		\coordinate (\j) at (\i,4);}
		\foreach \i/\j in {-7.5/9,-4.5/10,-1.5/11,1.5/12,4.5/13,7.5/14}{
		\coordinate (\j) at (\i,6);}
		\foreach \i/\j in {-6/15,-3/16,0/17,3/18,6/19}{
		\coordinate (\j) at (\i,8);}
		\foreach \i/\j in {-3/20,0/21,3/22}{
		\coordinate (\j) at (\i,10);}
		\coordinate (23) at (0,12);
		\draw[color=gray] (0) -- (1);
		\draw[color=gray] (0) -- (2);
		\draw[color=gray] (0) -- (3);
		\draw[color=gray] (1) -- (4);
		\draw[color=gray] (1) -- (6);
		\draw[color=gray] (2) -- (5);
		\draw[color=gray] (2) -- (7);
		\draw[color=gray] (4) -- (9);
		\draw[color=gray] (4) -- (10);
		\draw[ultra thick, color=purple] (5) -- (9);
		\draw[ultra thick, color=purple] (5) -- (12);
		\draw[ultra thick, color=purple] (9) -- (15);
		\draw[color=gray] (3) -- (6);
		\draw[color=gray] (3) -- (8);
		\draw[color=gray] (6) -- (11);
		\draw[color=gray] (7) -- (12);
		\draw[color=gray] (7) -- (14);
		\draw[color=gray] (8) -- (13);
		\draw[color=gray] (8) -- (14);
		\draw[color=gray] (10) -- (15);
		\draw[color=gray] (10) -- (16);
		\draw[color=gray] (11) -- (16);
		\draw[color=gray] (11) -- (18);
		\draw[ultra thick, color=purple] (12) -- (17);
		\draw[ultra thick, color=black] (13) -- (18);
		\draw[ultra thick, color=black] (13) -- (19);
		\draw[color=gray] (14) -- (19);
		\draw[ultra thick, color=purple] (15) -- (20);
		\draw[color=gray] (16) -- (21);
		\draw[ultra thick, color=purple] (17) -- (20);
		\draw[color=gray] (17) -- (22);
		\draw[ultra thick, color=black] (18) -- (21);
		\draw[ultra thick, color=black] (19) -- (22);
		\draw[color=gray] (20) -- (23);
		\draw[ultra thick, color=black] (21) -- (23);
		\draw[ultra thick, color=black] (22) -- (23);
		\foreach \i/\j in {0/1234, 1/1243, 2/1324, 3/2134, 4/1423, 5/1342, 6/2143, 7/3124, 8/2314}{
		\node[below] at (\i) {\scriptsize \j};} 
		\foreach \i/\j in {9/1432, 10/4123, 11/2413, 12/3142, 13/2341, 14/3214}{
		\node[below] at (\i) {\scriptsize \j};} 
		\foreach \i/\j in {15/4132, 16/4213, 17/3412, 18/2431, 19/3241, 20/4312, 21/4231, 22/3421, 23/4321}{
		\node[above] at (\i) {\scriptsize \j};} 
		\foreach \i in {0,1,...,23}{
		\filldraw[color=gray] (\i) circle (4pt);}
		\foreach \i in {5,9,12,15,17,20}{
		\filldraw[color=black] (\i) circle (4pt);}
		\foreach \i in {13,18,19,21,22,23}{
		\filldraw[color=black] (\i) circle (4pt);}
	\end{tikzpicture}
	\end{center}
	\caption{The induced subgraphs $\Gamma_{w,h}$ (in red) and $\Gamma_{\tilde{w},h}$ (in black) of $\Gamma_{h}$ for $h=(2,3,4,4)$ and $w=1342$.}
	\label{fig:2344}
\end{figure}

\begin{theorem}\label{thm:irregular}
For a Hessenberg function $h$ and $w\in\gen$, if $w$ contains one of associated patterns in the set $\{ \hpat{2143},  \hpat{1324}, \hpat{1243},\hpat{2134}, \hpat{1423}, \hpat{2314}, \hpat{2413} \}$, then the graph $\Gamma_{w,h}$ is irregular. 
\end{theorem}

\begin{proof}
By Lemma~\ref{lem:injection}, the map $\phi_{uv}\colon E_{w,h}(u)\to E_{w,h}(v)$ is injective for every edge $\{u,v\}\in \Gamma_{w,h}$ with $u\prec v$. For a permutation $w$ containing an associated pattern in Definition~\ref{def:pattern4}, to show that $\Gamma_{w,h}$ is irregular, we find an appropriate edge $\{u,v\}$ of $\Gamma_{w,h}$ with $u\prec v$ such that $\phi_{uv}$ is not surjective.

\textbf{Case $\hpat{2143}$.} Suppose that $w(j)<w(i)<w(\ell)<w(k)$ for some 
$i<j<k<\ell\leq h(i)$. 
Set $u\coloneqq w(i,k)(j,k)$ and $v\coloneqq u(k,\ell)$. Note that $(i,k)\in E_{w,h}(v)$ since $k\leq h(i)$ and
$$
v(i,k)=w(j,\ell)(i,j)\succ w(j,\ell)\succ w\,.
$$
In this case, if $\phi_{uv}(\tau)=(i,k)$, then $\tau\in \{ (i,k), (i,\ell) \}$. However, $u(i,k)=w(i,j)\prec w$ and $u(i,\ell)\not \succeq w$ since
$$
(u(i,\ell))[k]\!\!\uparrow=\{w(1),w(2),\dots,w(k-1),w(\ell)\}\!\!\uparrow < w[k]\!\!\uparrow.
$$
Therefore $(i,k), (i,\ell)\not \in E_{w,h}(u)$ and $\phi_{uv}$ is not surjective. 

\textbf{Case $\hpat{1324}$.} Suppose that $w(i)<w(k)<w(j)<w(\ell)$ for some $i<j<k<\ell\leq h(j)$ with $k\leq h(i)$.
Set $u\coloneqq w(k,\ell)$ and $v\coloneqq u(i,k)$. Note that $(j,\ell)\in E_{w,h}(v)$ since 
$\ell \leq h(j)$ and 
$$
v(j,\ell)=w(i,k)(j,\ell)(i,j)\succ w(i,k)(j,\ell)\succ w(i,k) \succ w\,.
$$ 
By the definition of $\phi_{uv}$, if $\phi_{uv}(\tau)=(j,\ell)$, then $\tau$ must be $(j,\ell)$. However,
$$
(u(j,\ell))[j]\!\!\uparrow=\{w(1),w(2),\dots,w(j-1),w(k)\}\!\!\uparrow < w[j]\!\!\uparrow,
$$
and this yields that $u(j,\ell)\not\succeq w$. Hence $(j,\ell)\not \in E_{w,h}(u)$ and $\phi_{uv}$ is not surjective. 

\textbf{Case $\hpat{1243}$.} Suppose that $w(i)<w(j)<w(\ell)<w(k)$ for some $i<j<k<\ell\leq h(j)$ with $j\leq h(i)<\ell$. Set $u\coloneqq w(j,k)$ and $v\coloneqq u(i,j)$. Note that $(j,\ell)\in E_{w,h}(v)$ since 
$\ell\leq h(j)$ and $v(j,\ell) \succ v \succeq w$. In this case, if $\phi_{uv}(\tau)=(j,\ell)$, then $\tau$ must be $(j,\ell)$ since $h(i)<\ell$. However, $u(j,\ell)\not \succeq w$ since
$$
(u(j,\ell))[k]\!\!\uparrow=\{w(1),w(2),\dots,w(k-1),w(\ell)\}\!\!\uparrow < w[k]\!\!\uparrow.
$$
So $(j,\ell)\not \in E_{w,h}(u)$ and $\phi_{uv}$ is not surjective. 

\textbf{Case $\hpat{2134}$.} Suppose that $w(j)<w(i)<w(k)<w(\ell)$ for some $i<j<k<\ell\leq h(k)$ with $k\leq h(i)<\ell$. Set $u\coloneqq w(j,k)$ and $v\coloneqq u(k,\ell)$. Note that $(i,k)\in E_{w,h}(v)$ since 
$k\leq h(i)$ and $v(i,k) \succ v \succeq w$. Similarly to the above case, if $\phi_{uv}(\tau)=(i,k)$, then $\tau$ must be $(i,k)$ since $h(i)<\ell$. However, $u(i,k)\not \succeq w$ since
$$
(u(i,k))[i]\!\!\uparrow=\{w(1),w(2),\dots,w(i-1),w(j)\}\!\!\uparrow < w[i]\!\!\uparrow.
$$
Thus $(i,k)\not \in E_{w,h}(u)$ and $\phi_{uv}$ is not surjective. 

\textbf{Case $\hpat{1423}$.} Suppose that $w(i)<w(k)<w(\ell)<w(j)$ 
for some $i<j<k<\ell\leq h(j)$ with $k\leq h(i)<\ell$. Set $u\coloneqq w(i,k)$ and $v\coloneqq u(i,j)$. Note that $(j,\ell)\in E_{w,h}(v)$ since 
$\ell \leq h(j)$ and $v(j,\ell) \succ v \succeq w$. Similarly to the above case, if $\phi_{uv}(\tau)=(j,\ell)$, then $\tau$ must be $(j,\ell)$ since $h(i)<\ell$. However, $u(j,\ell)\not \succeq w$ since
$$
(u(j,\ell))[k]\!\!\uparrow=((w[k]-\{w(j)\})\cup\{w(\ell)\})\!\!\uparrow < w[k]\!\!\uparrow.
$$
Thus $(j,\ell)\not \in E_{w,h}(u)$ and $\phi_{uv}$ is not surjective. 

\textbf{Case $\hpat{2314}$.} Suppose that $w(k)<w(i)<w(j)<w(\ell)$ for some $i<j<k<\ell\leq h(j)$ with $k\leq h(i)<\ell$. Set $u\coloneqq w(j,\ell)$ and $v\coloneqq u(k,\ell)$. Note that $(i,k)\in E_{w,h}(v)$ since 
$k \leq h(i)$ and $v(i,k) \succ v \succeq w$. Again, if $\phi_{uv}(\tau)=(i,k)$, then $\tau$ must be $(i,k)$ since $h(i)<\ell$. However, $u(i,k)\not \succeq w$ since
$$
(u(i,k))[i]\!\!\uparrow=\{w(1),w(2),\dots,w(i-1),w(k)\}\!\!\uparrow < w[i]\!\!\uparrow.
$$
Thus $(i,k)\not \in E_{w,h}(u)$ and $\phi_{uv}$ is not surjective. 

\textbf{Case $\hpat{2413}$.} Suppose that $w(k)<w(i)<w(\ell)<w(j)$ for some $i<j\leq h(i)<k\leq h(j)<\ell\leq h(k)$. Set $u\coloneqq w(i,j)$ and $v\coloneqq u(k,\ell)$. Note that $(j,k)\in E_{w,h}(v)$ since 
$k \leq h(j)$ and $v(j,k) \succ v \succeq w$. Again, if $\phi_{uv}(\tau)=(j,k)$, then $\tau$ must be $(j,k)$ since $h(j)<\ell$. However, $u(j,k)\not \succeq w$ since
$$
(u(j,k))[j]\!\!\uparrow=((w[j]-\{w(i)\})\cup\{w(k)\})\!\!\uparrow < w[j]\!\!\uparrow.
$$
Thus $(j,k)\not \in E_{w,h}(u)$ and $\phi_{uv}$ is not surjective. This completes the proof.

\end{proof}

As a consequence of Theorem~\ref{thm:irregular}, we can provide a necessary condition for $\Omega_{w}\cap\Hess(S,h)$ to be smooth when $w\in\gen$.

\begin{theorem}\label{thm:not-smooth_generator}
    For a Hessenberg function $h$, if $w$ is a permutation in $\gen$ containing one of the associated patterns in the set $\{ \hpat{2143}, \hpat{1324}, \hpat{1243},\hpat{2134}, \hpat{1423}, \hpat{2314}, \hpat{2413} \}$, then the intersection $\Omega_{w}\cap\Hess(S,h)$ is not smooth.
\end{theorem}
\begin{proof}
    If $w$  is a permutation in $\gen$ containing one of the associated patterns in the set above, then $\Gamma_{w,h}$ is an irregular graph by Theorem~\ref{thm:irregular}. 
    Therefore, $\Omega_{w}\cap\Hess(S,h)$ is not smooth by Propositions~\ref{prop:Hessenberg Schubert intersection} and~\ref{prop:h-interval}.
\end{proof}


\section{Regularity of the graphs \texorpdfstring{$\Gamma_{w,h}$}{Gwh}}\label{sec:regular}

In this section, we aim to characterize the regularity of $\Gamma_{w,h}$ in terms of pattern avoidance. First, we focus on the permutations in $\gen$ for the given Hessenberg function~$h$ and extend our work described in Theorem~\ref{thm:irregular} by considering the converse of the theorem. The permutations in $\gen$ avoiding the seven patterns have nice properties that ensure the regularity of~$\Gamma_{w,h}$. Subsequently, we assess more patterns to deal with an arbitrary permutation. The fact that for each $w$ there exists a unique permutation $\widetilde{w}\in\gen$ satisfying~\eqref{eq:generator}, where $\Gamma_{w,h}$ and $\Gamma_{\widetilde{w},h}$ are isomorphic graphs, plays an essential role.

\subsection{Well-aligned permutations in \texorpdfstring{$\gen$}{gen}}\label{sec:organized}

For a permutation $w\in\mathfrak{S}_n$, we define
$$
Y(w)\coloneqq \{w(i) \mid i\ge w^{-1}(1) \text{ and } w(i)\le w(n)\}.
$$
We always have $1 \in Y(w)$, and we denote the elements of $Y(w)$ as
$Y(w) = \{y_0, y_1, \dots, y_r\}$ with $y_0=1 < y_1 < \cdots < y_r=w(n)$, 
where it is possible that $y_0 = y_r$ (i.e., $Y(w)$ may consist of a single element, namely $1$).
The following lemma shows some properties of permutations in $\gen$.

\begin{lemma}\label{lem:y}
For a Hessenberg function $h$, let $w$ be a permutation in $\gen$. 
\begin{enumerate}
\item If $h(i)<j$ and $w(i)<w(j)$, then there exist $x$ and $z$ with $i<x\leq h(i)<j$ and $i<z<j\leq h(z)$ such that $w(i)<w(x)$ and $w(z)<w(j)$. 
\item If $ Y(w)=\{y_0,y_1,\dots,y_{r}\}$, then $w^{-1}(y_i)\leq h(w^{-1}(y_{i-1}))$ for $1\leq i \leq r$.
\end{enumerate}
\end{lemma}

\begin{proof}
Recall that $w^{-1}(a+1)\leq h(w^{-1}(a))$ for $1\leq a\leq n-1$ since $w\in\gen$.
\begin{enumerate}
\item Let $w(x)$ be the smallest integer satisfying that $i<x$ and $w(i)<w(x)$. Then 
$i<x\leq h(w^{-1}(w(x)-1))\leq h(i)<j$ and $w(x)<w(j)$ since $w(x)$ is the smallest and $w\in \gen$.
Similarly, let $w(z)$ be the largest integer satisfying that $z<j$ and $w(z)<w(j)$. Then 
$i<z<j\leq w^{-1}(w(z)+1)\leq h(z)$ and $w(i)<w(z)$ since $w(z)$ is the largest and $w\in \gen$. 
\item If $y_{i}=y_{i-1}+1$, then $w^{-1}(y_i)=w^{-1}(y_{i-1}+1)\leq h(w^{-1}(y_{i-1}))$. On the other hand, if $y_i > y_{i-1}+1$, then $y_{i}-1\not \in Y(w)$ and $w^{-1}(y_{i}-1)<w^{-1}(1)\leq w^{-1}(y_{i-1})$ so that 
$$
w^{-1}(y_i)\leq h(w^{-1}(y_i-1))\leq h(w^{-1}(1))\leq h(w^{-1}(y_{i-1}))\,.
$$ 
\end{enumerate}
This proves the lemma.
\end{proof}

Now, we define the notion of well-aligned permutations and study their properties. We say that $w \in \mathfrak{S}_n$ is a \emph{well-aligned} permutation with $Y(w)=\{y_0,y_1,\dots,y_{r}\}$ if it satisfies that $w^{-1}(y_0)<w^{-1}(y_1)<\cdots<w^{-1}(y_{r})=n$.

For a well-aligned permutation $w$ with $Y(w)=\{y_0,\dots,y_r\}$, 
we define a sequence of permutations $\overline{w}_0, \overline{w}_1, \dots, \overline{w}_r$ 
associated to $w$ by
\[
\overline{w}_0 \coloneqq w, \quad
\overline{w}_m \coloneqq (1,y_m)\,\overline{w}_{m-1} \quad (m=1,\dots,r).
\]
We stress that the notation $\overline{w}_0$ is entirely unrelated to $w_0$, 
the longest permutation, and should not be confused with it.
For the sake of simplicity, we let $\overline{w}\coloneqq \overline{w}_{r}$.

\begin{example}
Consider the two permutations $w=213465$ and $w'=214365$. 
For $w=213465$, we have $Y(w)=\{1,3,4,5\}$, and these elements appear 
in increasing order in the one-line notation. Thus $w$ is well-aligned,
and we obtain the chain
\[
\overline{w}_0=213465 \prec \overline{w}_1=231465 \prec 
\overline{w}_2=234165 \prec \overline{w}=234561 \,.
\]
In contrast, for $w'=214365$, although $Y(w')$ is the same as $Y(w)$,
$w'$ is not well-aligned since 4 precedes 3 in $w'$.
\end{example}

The following proposition shows several properties of $\overline{w}$ for a well-aligned permutation $w\in\gen$.  

\begin{proposition}\label{prop:wbar}
For a Hessenberg function $h$ on $[n]$, if $w\in\gen$ is well-aligned with $Y(w)=\{y_0,y_1,\dots,y_{r}\}$, then the following hold.

\begin{enumerate}
\item $\overline{w}\in\gen$.
\item $[\overline{w},w_0]=\{u\in [w,w_0] \mid u(n)=1\}$.
\item $\{(i,n)\in E_{w,h}(\overline{w}) \}= \{(w^{-1}(y_s),n) \mid h(w^{-1}(y_s))=n \text{~for~} 0\leq s < r \}$.
\end{enumerate}

\end{proposition}

\begin{proof}

Note that $\overline{w}^{-1}(a)=w^{-1}(a)$ for all $a\not\in Y(w)$.

\begin{enumerate}
\item 
Since $w\in\gen$, 
for $1\leq a\leq n-1$, 
if $a,a+1\not\in Y(w)$, then 
$$
\overline{w}^{-1}(a+1)=w^{-1}(a+1)\leq h(w^{-1}(a))=h(\overline{w}^{-1}(a))\,.
$$
Hence, it suffices to show that $\overline{w}^{-1}(a+1)\leq h(\overline{w}^{-1}(a))$ for $1\leq a\leq n-1$ such that $a\in Y(w)$ or $a+1\in Y(w)$. First, consider the case $a\not \in Y(w)$ and $a+1 \in Y(w)$. Let $y_s-1\not \in Y(w)$ for some $1\leq s \leq r$. 
Since $w\in\gen$ is well-aligned, we have 
$$
\overline{w}^{-1}(y_s)=w^{-1}(y_{s-1}) <w^{-1}(y_s) \leq h(w^{-1}(y_s-1))=h(\overline{w}^{-1}(y_s-1))\,.
$$
Now consider the case $a=y_s$ for $0\leq s\leq r$. If $r=0$, then $\overline{w}^{-1}(2)\leq h(\overline{w}^{-1}(1))=n$. If $s=r$, then by Lemma~\ref{lem:y}~(2) we have 
$$
\overline{w}^{-1}(y_{r}+1)< n =w^{-1}(y_{r})\leq h(w^{-1}(y_{r-1}))=h(\overline{w}^{-1}(y_{r}))\,.
$$
For $0<s<r$, if $y_s+1\in Y(w)$, then $y_s+1=y_{s+1}$ so that $\overline{w}^{-1}(y_s+1)=w^{-1}(y_s)$. From Lemma~\ref{lem:y}~(2) it follows that
$$
\overline{w}^{-1}(y_s+1)=w^{-1}(y_s) \leq h(w^{-1}(y_{s-1}))=h(\overline{w}^{-1}(y_s))\,.
$$
On the other hand, if $y_s+1\not \in Y(w)$, then we have
$$
\overline{w}^{-1}(y_s+1)=w^{-1}(y_s+1)<w^{-1}(1)<w^{-1}(y_s)\leq h(w^{-1}(y_{s-1}))=h(\overline{w}^{-1}(y_s))\,.
$$
\item Since $[\overline{w},w_0]\subseteq\{u\in [w,w_0]\mid u(n)=1\}$, it suffices to show that $\overline{w}\preceq u$ for each $u\in\Sn{n}$ with $w\preceq u$ and $u(n)=1$.
Suppose that $u[k]\!\!\uparrow< \overline{w}[k]\!\!\uparrow$ for some $k$. Since $w[m]=\overline{w}[m]$ for $m<w^{-1}(1)$ or $m=n$, it follows that $w^{-1}(1)\leq k<n$ and $\overline{w}[k]=(w[k]-\{1\})\cup\{y_s\}$ for some $s>0$. More precisely, since $w$ is well-aligned, if we denote 
$w[k]\!\!\uparrow$, $u[k]\!\!\uparrow$, and $\overline{w}[k]\!\!\uparrow$ by ordered sets $\{a_1,a_2,\dots,a_k\}$,
$\{b_1,b_2,\dots,b_k\}$, and $\{c_1,c_2,\dots,c_k\}$, respectively, then $a_i=i$ and $c_i=i+1$ for $1\leq i<y_s$, and $a_i=c_i$ for $y_s\leq i\leq k$. Note that $w\preceq u$ implies that $a_i\leq b_i$ for $1\leq i \leq k$. Hence, if $a_i\leq b_i<c_i$ for some $i$, then $i<y_s$ so that $i\leq c_i<i+1$, or equivalently, $\{c_1,c_2,\dots,c_i\}=[i]$. This contradicts that $u(n)=1$ and $k<n$. Therefore, there is no such $k$.
\item We first check the inclusion ($\subseteq$). Let $(i,n)\in E_{w,h}(\overline{w})$. Then $h(i)=n$ and $\overline{w}(i)<y_{r}$ since $\overline{w}(i,n)\succeq w$ and 
$(\overline{w}(i,n))[n-1]=(w[n-1]-\{\overline{w}(i)\})\cup\{y_{r}\}$.
Note that $i<w^{-1}(1)$ yields that $(\overline{w}(i,n))[i]=w[i-1]\cup\{1\}$ and $\overline{w}(i,n)\not\succeq w$. Hence, $i\geq w^{-1}(1)$ so that $\overline{w}(i)=y_{s+1}$, or equivalently $w(i)=y_s$,  for some $0\leq s< r$. Therefore $(i,n)=(w^{-1}(y_s),n)$ and $h(w^{-1}(y_s))=n$.

Now we check the inclusion ($\supseteq$). Assume that $w(i)=y_s$ and $h(i)=n$ for some $i<n$.  
By Lemma~\ref{lemma:Bruhat order} $\overline{w}(i,n)\succeq w$ because
$\overline{w}(i,n)$ and $w$ agree everywhere except on positions $w^{-1}(y_0), w^{-1}(y_1), \dots, w^{-1}(y_{r})$, and on these positions $\overline{w}(i,n)$ and $w$ have the subsequences $y_1\dots y_s y_0 y_{s+2} \dots y_{r} y_{s+1}$ and $y_0 \dots y_{s-1} y_s y_{s+1} \dots y_{r-1}y_{r}$, respectively.
Therefore $(w^{-1}(y_s),n)\in E_{w,h}(\overline{w})$.
\end{enumerate}
\end{proof}

For a well-aligned permutation $w$, we say that $w$ is 
\emph{bottom-packed} (respectively, \emph{right-packed}) if 
$y_i = i+1$ (respectively, $w(n-i)=y_{r-i}$) for $0\le i \le r$. 
For example, the permutations $461523$ and $426135$ are bottom-packed 
and right-packed, respectively, while $213654$ is a well-aligned 
permutation that is neither bottom-packed nor right-packed. 
In the following, we show that the map $\phi_{\overline{w}_{m-1}\overline{w}_{m}}$ is a bijection if $w\in\gen$ is a well-aligned permutation that satisfies a certain condition.

\begin{proposition}\label{prop:organized}
For a Hessenberg function $h$ on $[n]$, let $w\in\gen$ be a well-aligned permutation with $Y(w)=\{y_0,y_1,\dots,y_{r}\}$ for some $r\geq 1$. For $1\leq m \leq r$, let $\overline{w}_m=\overline{w}_{m-1}(a,b)$. If $h(p)\geq b$ for all 
$(p,a)\in E_{w,h}(\overline{w}_m)$ such that $p<w^{-1}(1)$, then
$$
\phi_{\overline{w}_{m-1}\overline{w}_{m}}\colon E_{w,h}(\overline{w}_{m-1}) \to E_{w,h}(\overline{w}_{m})
$$
is a bijection and $\phi_{\overline{w}_{m-1}\overline{w}_{m}}(i,j)=(i,j)$ except 
$\phi_{\overline{w}_{m-1}\overline{w}_{m}}(p,b)=(p,a)$ for $p<w^{-1}(1)$.

In particular, if $w\in\gen$ is bottom-packed,
then $E_{w,h}(\overline{w}_{m-1})=E_{w,h}(\overline{w}_{m})$ and $\phi_{\overline{w}_{m-1}\overline{w}_{m}}$ is an identity map.
\end{proposition}

\begin{proof}
For each $m$, the transposition $(a,b)$ in the statement refers to $(w^{-1}(y_{m-1}), w^{-1}(y_m))$, 
so it depends on $m$. This agrees with the definition of $\overline{w}_m$, as $\overline{w}_m = (1,y_m)\,\overline{w}_{m-1}$.
Since $w\in\gen$ we have $(a,b)\in 
E_{w,h}(\overline{w}_{m-1})$ by Lemma~\ref{lem:y}~(2), and thus $\phi_{\overline{w}_{m-1}\overline{w}_{m}}$ is injective by Lemma~\ref{lem:injection}. So it suffices to check the surjectivity of $\phi_{\overline{w}_{m-1}\overline{w}_{m}}$. 
Let 
$$
A \coloneqq E_{w,h}(\overline{w}_{m})- \{(p,a)\mid p<w^{-1}(1)\}\,.
$$
We show that 
$\phi_{\overline{w}_{m-1}\overline{w}_{m}}(i,j)=(i,j)$
for $(i,j)\in A$
and then find the preimage of $(p,a)$ under $\phi_{\overline{w}_{m-1}\overline{w}_{m}}$
for each
$(p,a)\in E_{w,h}(\overline{w}_{m})$
with $p<w^{-1}(1)$.
Note that since $\overline{w}_m=(1,y_m)\overline{w}_{m-1}=\overline{w}_{m-1}(a,b)$, we get $\overline{w}_{m-1}(a)=1$ and $\overline{w}_{m-1}(b)=y_m$.
 
We first show that $A\subseteq E_{w,h}(\overline{w}_{m-1})$.
Let $(i,j)\in A$. It suffices to show that $\overline{w}_{m-1}(i,j)\succeq w$. Note that if $\overline{w}_{m-1}(i)<\overline{w}_{m-1}(j)$, then $\overline{w}_{m-1}(i,j)\succ \overline{w}_{m-1} \succeq w$. Now we assume that $\overline{w}_{m-1}(i)>\overline{w}_{m-1}(j)$. 
Then $i \neq a$ and $i \neq b$ since $\overline{w}_{m-1}(a)=1$ and 
$\overline{w}_{m-1}(b)=y_m=\min\overline{w}_{m-1}[a+1,n]$, respectively.
Note that if $i<b$ and $j>a$, then
$$
(\overline{w}_{m}(i,j))[k]=(w[k]-\{1,\overline{w}_{m}(i)\})\cup\{y_m,w(j)\}
$$
for $\max\{a,i\}\leq k <\min\{b,j\}$.
It follows from $\overline{w}_{m}(i,j)\succeq w$ that $\overline{w}_{m}(i)<w(j)$. However, this contradicts $w(i)>w(j)$ since $w(i) \leq \overline{w}_{m}(i)$.  
Hence 
it remains the case when $i>b$ or $j\leq a$.
For the remaining $(i,j)$, we split them into two cases as follows.

\begin{itemize}
\item 
Let $i>b$ or $j<a$. Then $\overline{w}_{m-1}(i,j)\succeq w$ since
$$
(\overline{w}_{m-1}(i,j))[k]=\begin{cases}
w[k] & \mbox{if $a\leq k<b$},\\
(\overline{w}_{m}(i,j))[k] & \mbox{otherwise}.
\end{cases}
$$
\item
Let $j=a$ and $i\geq w^{-1}(1)$. If $w(i)>y_m$, then
\[
(\overline{w}_{m}(i,a))[i]=(w[i-1]-\{1\})\cup\{y,y_m\} 
\]
for some $y<y_m$, which contradicts  $\overline{w}_{m}(i,a)\succeq w$. Therefore, $w(i)<y_m$ so that $w(i)=y_s$ for some $s<m-1$. 
By Lemma~\ref{lemma:Bruhat order} $\overline{w}_{m-1}(i,a)\succeq w$ since
$\overline{w}_{m-1}(i,a)$ and $w$ agree everywhere except on positions $w^{-1}(y_0), w^{-1}(y_1), \dots, w^{-1}(y_{m-1})$, and on these positions $\overline{w}_{m-1}(i,a)$ and $w$ have the subsequences $y_1\dots y_s y_0 y_{s+2} \dots y_{m-1} y_{s+1}$ and $y_0 \dots y_{s-1} y_s y_{s+1} \dots y_{m-2}y_{m-1}$, respectively.
\end{itemize}
Therefore, $A\subseteq E_{w,h}(\overline{w}_{m-1})$.

Now consider the sets $P\coloneqq \{(p,j)\mid p<a,~j\in\{a,b\}\}$ and $Q\coloneqq \{(i,q)\mid i\in\{a,b\},~q>b\}$.

\begin{itemize}
\item Suppose that $(i,j)\in A-(P\cup Q)$. It follows from the definition of the map that $\phi_{\overline{w}_{m-1}\overline{w}_{m}}$ fixes all $(i,j)$.
\item Suppose that $(i,q)\in Q$. Note that $\overline{w}_{m}(i,q)\succeq w$
since $\overline{w}_{m}(a)=y_m<\overline{w}_{m}(q)$ and $\overline{w}_{m}(b)=1$.
This yields that if $h(a)\geq q$, then $(a,q),(b,q)\in A$; if $h(a)< q \leq h(b)$, then $(a,q)\not \in A$, $(b,q) \in A$, and $(a,q)\not \in E_{w,h}(\overline{w}_{m-1})$. In any case, we have $\phi_{\overline{w}_{m-1}\overline{w}_{m}}(i,q)=(i,q)$ for $(i,q)\in A\cap Q$.
\item Suppose that $(p,j)\in P$. Note that if $p<w^{-1}(1)$, then $\overline{w}_{m}(p,b)\not \succeq w$ since
$(\overline{w}_m(p,b))(p)=1$. Thus if $(p,b)\in A$, then $w^{-1}(1)\leq p<a$ and $(\overline{w}_m(p,b))(b)=\overline{w}_m(p)\leq w(b)=y_m$ so that $\overline{w}_m(p)=y_s$ for some $s<m$ since $w$ is well-aligned. This yields that $(p,a), (p,b)\in A$, and $\phi_{\overline{w}_{m-1}\overline{w}_{m}}(p,j)=(p,j)$ for $(p,j)\in A\cap P$. 

\end{itemize}
This shows that
$\phi_{\overline{w}_{m-1}\overline{w}_{m}}(i,j)=(i,j)$
for all $(i,j)\in A$.

Now we assume that $(p,a) \in E_{w,h}(\overline{w}_m)$ and $p<w^{-1}(1)$ and find the preimage of $(p,a)$ under $\phi_{\overline{w}_{m-1}\overline{w}_{m}}$.
By the definition of $\phi_{\overline{w}_{m-1}\overline{w}_{m}}$,
if $\phi_{\overline{w}_{m-1}\overline{w}_{m}}(\tau)=(p,a)$ for some $\tau$, then $\tau \in \{(p,a), (p,b)\}$. 
Note that $(p,a)\in E_{w,h}(\overline{w}_m)$ implies that $w(p)<y_m$,
and $h(p)\geq b$ from the assumption. Since $\overline{w}_{m-1}(a)=1$ and $\overline{w}_{m-1}(b)=y_m$, we have $\overline{w}_{m-1}(p,a)\not\succeq w$ and $\overline{w}_{m-1}(p,b)\succeq w$, respectively. Therefore, $(p,a)\not\in E_{w,h}(\overline{w}_{m-1})$ and $(p,b)\in E_{w,h}(\overline{w}_{m-1})$, so 
$\phi_{\overline{w}_{m-1}\overline{w}_{m}}(p,b)=(p,a)$
when $p<w^{-1}(1)$. Hence the map $\phi_{\overline{w}_{m-1}\overline{w}_{m}}$ is bijective.

Note that if $w\in\gen$ is bottom-packed, then $w(p)>y_m$ for all $p<w^{-1}(1)$. Hence, $E_{w,h}(\overline{w}_{m})=A$ and the map $\phi_{\overline{w}_{m-1}\overline{w}_{m}}$ is an identity map.
This proves the proposition.
\end{proof}

\subsection{Regularity of \texorpdfstring{$\Gamma_{w,h}$}{Gwh} for pattern avoiding permutations in \texorpdfstring{$\gen$}{gen}}\label{sec:regularproof}

In this subsection, 
we first investigate the relationship between the properties of well-aligned permutations and those of pattern avoiding permutations in $\gen$.
We then show that the regularity of $\Gamma_{w,h}$, for $w\in \gen$, can be characterized by the pattern avoidance of $w$.

For simplicity, we assume that $Y(w)=\{y_0,y_1,\dots,y_r\}$ in this subsection unless otherwise stated.

To establish the regularity of $\Gamma_{w,h}$, we begin by examining how well-aligned permutations relate to pattern avoidance in $\gen$.

\begin{lemma}\label{lem:1324}
    Let $w$ be a permutation in $\gen$ for a Hessenberg function~$h$. If $w$ avoids the associated pattern $\hpat{1324}$, then $w$ is well-aligned. 
\end{lemma}

\begin{proof}
    Suppose that $w$ is not well-aligned so that we have 
$w^{-1}(y_{i})<w^{-1}(y_j)<w^{-1}(y_{i+1})$ for some $i$ and $j$ with $0<i+1<j<r$.  
Let $k$ be the smallest integer such that $k>j$ and $w^{-1}(y_{i+1})<w^{-1}(y_k)$.
Note that if $w^{-1}(y_{k-1})<w^{-1}(y_i)$, then the subsequence $y_i y_j y_{i+1} y_k$ gives the associated pattern $\hpat{1324}$ since $w^{-1}(y_{k})\leq h(w^{-1}(y_{k-1}))\leq h(w^{-1}(y_i))$ by Lemma~\ref{lem:y}~(2). On the other hand, if $w^{-1}(y_i)<w^{-1}(y_{k-1})$, then the subsequence $y_i y_{k-1} y_{i+1} y_k$ gives the associated pattern $\hpat{1324}$ since $w^{-1}(y_{i+1})\leq h(w^{-1}(y_{i}))$ and $w^{-1}(y_{k})\leq h(w^{-1}(y_{k-1}))$ by Lemma~\ref{lem:y}~(2). Hence, $w$ always contains the associated pattern $\hpat{1324}$.
\end{proof}

Note that for $w\in\gen$, avoiding the associated pattern $\hpat{1324}$ is not a necessary condition to be well-aligned. For instance, when $h=(7,7,7,7,7,7,7)$, the permutation $4651273$ is well-aligned but contains the pattern $1324$.

If $w\in\gen$ is well-aligned and avoids certain patterns in Definition~\ref{def:pattern4}, then it is bottom-packed or right-packed.  
Next, we refine the structure of well-aligned permutations by introducing additional pattern avoidance conditions, which enables us to identify when a permutation is bottom-packed or right-packed—a crucial step toward computing the edges of $\Gamma_{w,h}$.

\begin{lemma}\label{lem:bottom-right-packed}
    For a Hessenberg function~$h$, assume that $w\in\gen$ is well-aligned.
    \begin{enumerate}
        \item If $h(w^{-1}(1))<n$ and $w$ avoids the associated pattern $\hpat{2134}$, then $w$ is bottom-packed. 
        In addition, if $w$ also avoids $\hpat{1243}$ and $\hpat{1423}$, then $w(i)\in Y(w)$ for all $i$ such that $h(i)=n$ .
        \item If $h(w^{-1}(1))=n$ and $w$ avoids the associated patterns $\hpat{2143}$ and $\hpat{2134}$, then $w$ is bottom-packed or right-packed.
    \end{enumerate}
\end{lemma}

\begin{proof}
\begin{enumerate}
\item 
First note that $y_{r}>2$ and $y_{r}-1\in Y(w)$, otherwise we have $h(w^{-1}(1))=n$.
Now suppose that $\{x+1,x+2,\dots,y_{r}\}\subset Y(w)$ and $x \not \in Y(w)$ for some $1<x<y_{r}$.
Since $w^{-1}(x+1)\leq h(w^{-1}(x))$ and $h(w^{-1}(1))<n$, we have $w^{-1}(x+j)\leq h(w^{-1}(x))<w^{-1}(x+j+1)$ for some $1\leq j<y_{r}-x$. Then the subsequence $x1(x+j)(x+j+1)$ gives the associated pattern $\hpat{2134}$ since $w^{-1}(x+j+1)\leq h(w^{-1}(x+j))$. Hence, there is no such $x$, and $w$ is bottom-packed. 

For the latter part, suppose that there is  $i$ satisfying that $h(i)=n$ and $w(i)>r+1=w(n)$. Let $x$ be the smallest integer satisfying that $1\leq x <r$ and $h(w^{-1}(x+1))=n$. Note that if $w^{-1}(x+1)<i$, then the subsequence $x (x+1) w(i) (r+1)$ gives the associated pattern $\hpat{1243}$. On the other hand, if $w^{-1}(x+1)>i$, then the subsequence $x\, w(i)\, (x+1)\,  (r+1)$ gives the associated pattern $\hpat{1423}$. Therefore there is no such $i$.
\item 
Suppose that $w$ is not bottom-packed so that $y_{s}-1\not \in Y(w)$ for some $1\leq s \leq r$. 
Now assume that there is $x$ satisfying that $x>w(n)$ and $w^{-1}(1)<w^{-1}(x)$. If $w^{-1}(x)<w^{-1}(y_s)$, then the subsequence $(y_s-1) 1 x y_s$ gives the associated pattern $\hpat{2143}$ since $w^{-1}(y_{s})\leq h(w^{-1}(y_{s}-1))$. Hence, $w^{-1}(y_s)<w^{-1}(x)$ 
and $s<r$. This yields that $h(w^{-1}(y_{s}-1))=n$, otherwise the subsequence $(y_s-1) 1 y_s y_r$ gives the associated pattern $\hpat{2134}$ since $h(w^{-1}(1))=n$. Then the subsequence $(y_s-1) 1 x y_r$ gives the associated pattern $\hpat{2143}$, which is a contradiction.
Therefore there is no such $x$, and $w$ is right-packed. 
\end{enumerate}
This completes the proof.
\end{proof}

With the bottom-packed or right-packed structure established, we can now compute the sizes of the sets 
$E_{w,h}(\overline{w})$ and $E_{w,h}(w_0)$, 
which will play a key role in verifying the regularity of $\Gamma_{w,h}$.

\begin{proposition}\label{prop:size of Ewh}
    For a Hessenberg function~$h$, assume that $w\in\gen$ avoids the three associated patterns $\hpat{2143}$, $\hpat{1324}$, and $\hpat{2134}$. Then the following hold.
    \begin{enumerate}
        \item If $w$ also avoids the associated pattern $\hpat{2314}$, then $|E_{w,h}(w)|=|E_{w,h}(\overline{w})|$.
\item If $w$ also avoids the two associated patterns $\hpat{1243}$  and $\hpat{1423}$, then 
$$
\{(i,n)\in E_{w,h}(w_0) \}= \{(i,n) \mid n-k+1 \leq i<n\},
$$
where $k=|\{y_i \mid h(w^{-1}(y_i))=n\}|$.
    \end{enumerate}
\end{proposition}
\begin{proof}
    From the assumption, $w$ is a well-aligned permutation in $\gen$ by Lemma~\ref{lem:1324}.
    \begin{enumerate}
        \item Note that from Proposition~\ref{prop:organized} if $w$ is bottom-packed, then we have 
$$
E_{w,h}(w)=E_{w,h}(\overline{w}_0)=E_{w,h}(\overline{w}_1)=\cdots= E_{w,h}(\overline{w}_{r})=E_{w,h}(\overline{w})\,.
$$ 
Hence we only need to consider well-aligned permutations $w$ which are not bottom-packed. Let $Y(w)\neq [r+1]$. It follows from Lemma~\ref{lem:bottom-right-packed} that $h(w^{-1}(1))=n$. 
Suppose that for some $1\leq m\leq r$, the injection $\phi_{\overline{w}_{m-1}\overline{w}_{m}}\colon E_{w,h}(\overline{w}_{m-1}) \to E_{w,h}(\overline{w}_{m})$ is not surjective. Let $\overline{w}_m=\overline{w}_{m-1}(a,b)$, where $w(a)=y_{m-1}$ and $w(b)=y_m$. Then there exists $p$ such that $p<w^{-1}(1)$, $(p,a)\in E_{w,h}(\overline{w}_m)$, and $a\leq h(p)<b$ by Proposition~\ref{prop:organized}. Note that $y_m=\overline{w}_m(a)=(\overline{w}_m(p,a))(p)$. Since $(p,a)\in E_{w,h}(\overline{w}_m)$, i.e., $\overline{w}_m(p,a)\succeq w$, and $w[p-1]=\overline{w}_m[p-1]$ we have $y_m>w(p)$.
If $w(p)<y_{m-1}$, then the subsequence $w(p)1y_{m-1}y_{m}$ gives the associated pattern $\hpat{2134}$. Hence $y_{m-1}<w(p)<y_{m}$ and $y_{m-1}\neq y_m-1$. Note that $p<w^{-1}(y_m-1)<w^{-1}(1)$ since $h(p)<w^{-1}(y_m)\leq h(w^{-1}(y_m-1))$. Then the subsequence $w(p)(y_m-1)y_{m-1}y_{m}$ gives the associated pattern $\hpat{2314}$. Thus there is no such $p$ for all $m=1,2,\dots,r$ so that 
$$
|E_{w,h}(w)|=|E_{w,h}(\overline{w}_0)|=|E_{w,h}(\overline{w}_1)|=\cdots= |E_{w,h}(\overline{w}_{r})|=|E_{w,h}(\overline{w})|\,.
$$  
\item Let $|\{y_i \mid h(w^{-1}(y_i))=n\}|=k$ and $|\{i \mid h(i)=n\}|=m$. Then $1\leq k \leq m$.
For $n-m+1\leq i<n$, note that $(w_0(i,n))[i]\!\!\uparrow \geq w[i]\!\!\uparrow$ if and only if $w^{-1}(1)\leq i$ since 
\begin{align*}
w[i]&=[n]-\{w(i+1),w(i+2),\dots,w(n)\},\\
(w_0(i,n))[i]&=[n]-\{2,3,\dots,n-i+1\}\,.
\end{align*}
In addition, $(w_0(i,n))[n-1]\!\!\uparrow \geq w[n-1]\!\!\uparrow$ if and only if $n-i+1\leq w(n)$.
From Proposition~\ref{prop:Bruhat order} (3), it follows that $w_0(i,n)\succeq w$ if and only if 
$(w_0(i,n))[i]\!\!\uparrow \geq w[i]\!\!\uparrow$ and $(w_0(i,n))[n-1]\!\!\uparrow \geq w[n-1]\!\!\uparrow$ since
$\{i,n-1\}=[n-1]-D(w_0(i,n))$. Hence 
$$
\{(i,n)\in E_{w,h}(w_0) \}= \{(i,n) \mid \max\{w^{-1}(1), n+1-w(n), n-m+1\} \leq i<n\}\,.
$$ 
If $h(w^{-1}(1))<n$, then $w^{-1}(1)<n-m+1$, $w(n)=r+1$, and $m=k\leq r$ by Lemma~\ref{lem:bottom-right-packed} (1). On the other hand, if $h(w^{-1}(1))=n$, then from Lemma~\ref{lem:bottom-right-packed} (2) it follows that $k=r+1$, $w^{-1}(1)=n-r=n-k+1$, and $w(n)>r$. 
In any case, we have 
$\max\{w^{-1}(1), n+1-w(n), n-m+1\}=n-k+1$, 
as we desired.
    \end{enumerate}
    This completes the proof.
\end{proof}

To proceed with the induction, we verify that $\overline{w}$
 inherits the pattern avoidance properties of $w$ in $\gen$.

\begin{proposition}\label{prop:chain_w}
Let $w$ be a permutation in $\gen$ for a Hessenberg function~$h$. If $w$ avoids all of the associated patterns $\hpat{2143}$, $\hpat{1324}$, $\hpat{1243}$, $\hpat{2134}$, $\hpat{1423}$, $\hpat{2314}$, and $\hpat{2413}$, then so does $\overline{w}$.
\end{proposition}

\begin{proof}
Since $w$ avoids the associated patterns $\hpat{2143}$, $\hpat{1324}$, and  $\hpat{2134}$, $w$ is bottom-packed or right-packed by Lemma~\ref{lem:bottom-right-packed}. Note that if $w$ is bottom-packed, then $w(i)<w(j)$ if and only if $\overline{w}(i)<\overline{w}(j)$ for all $1\leq i<j <n$. Moreover, since $\overline{w}(n)=1$ and none of the seven associated patterns ends with $1$, $\overline{w}$ avoids all of the associated patterns if $w$ is bottom-packed.

Now we consider the case when $w$ is right-packed. Then $h(w^{-1}(1))=n$ by Lemma~\ref{lem:bottom-right-packed}. In the following, we show that if $\overline{w}$ contains one of the associated patterns in Definition~\ref{def:pattern4}, then so does $w$. Note that $\overline{w}$ and $w$ may contain different associated patterns. For a given associated pattern, 
we suppose that $\overline{w}$ contains the subsequence $\overline{v}\coloneqq \overline{w}(i)\overline{w}(j)\overline{w}(k)\overline{w}(\ell)$ with $1\leq i<j<k<\ell\leq n$ that gives  the pattern. 
Since $w$ is right-packed and avoids all of the patterns, we may assume that the subsequence $\overline{v}$ can be written as one of the forms $w(i)y_{q+1} y_{s+1} y_{m+1}$, $w(i)w(j)y_{s+1} y_{m+1}$, and $w(i)w(j)w(k)y_{m+1}$ for some $0\leq q<s<m<r$.

\textbf{Case $\hpat{2143}$.} 
Suppose that $\overline{w}(j)<\overline{w}(i)<\overline{w}(\ell)<\overline{w}(k)$ for some 
$\ell\leq h(i)$. 
It suffices to consider the subsequence $\overline{v}=w(i)w(j)w(k)y_{m+1}$. Then $w^{-1}(y_m)=\ell\leq h(i)$, and we may assume that $y_m<w(i)$. 
Note that if $w^{-1}(y_{m+1})\leq h(i)$, then $w$ contains the subsequence $w(i)w(j)w(k) y_{m+1}$ that gives the associated pattern $\hpat{2143}$. On the other hand, if $w^{-1}(y_{m+1})> h(i)$, then $w^{-1}(y_{m+1}-1)>i$ and $w(i)(y_{m+1}-1) y_m y_{m+1}$ gives the associated pattern $\hpat{2314}$.

\textbf{Case $\hpat{1324}$.} 
Suppose that $\overline{w}(i)<\overline{w}(k)<\overline{w}(j)<\overline{w}(\ell)$, $k\leq h(i)$, and $\ell\leq h(j)$.
It suffices to consider the subsequence $\overline{v}=w(i)w(j)y_{s+1} y_{m+1}$ or $\overline{v}=w(i)w(j)w(k)y_{m+1}$.
First, let $\overline{v}=w(i)w(j)y_{s+1} y_{m+1}$. Then $w^{-1}(y_s)\leq h(i)$ and $w^{-1}(y_m) \leq h(j)$, and we may assume that $y_s<w(i)$ or $y_m<w(j)$. 
Note that if $w^{-1}(y_{m+1})\leq h(i)$, then $w$ contains $w(i)w(j) y_{s+1} y_{m+1}$ that gives the associated pattern $\hpat{1324}$. Now we assume that $w^{-1}(y_{m+1})> h(i)$, and consider three possible cases according to $y_s$ and $y_m$. 
\begin{itemize}
\item Let $y_s>w(i)$ and $y_m<w(j)$. If $w^{-1}(y_{m+1})\leq h(j)$, then $w$ contains $w(i)w(j) y_s y_{m+1}$ that gives the associated pattern $\hpat{1324}$. On the other hand, if $w^{-1}(y_{m+1})> h(j)$, then $w(j)(y_{m+1}-1) y_m y_{m+1}$ gives the associated pattern $\hpat{2314}$.
\item Let $y_s<w(i)$ and $y_m>w(j)$. If $w^{-1}(y_{m+1})\leq h(j)$, then $w(i)w(j) y_s y_{m+1}$ gives the associated pattern $\hpat{2314}$. On the other hand, if $w^{-1}(y_{m+1})> h(j)$, then $w(j)y_s y_m y_{m+1}$ gives the associated pattern $\hpat{2134}$.
\item Let $y_s<w(i)$ and $y_m<w(j)$. If $w^{-1}(y_{m+1})\leq h(j)$, then $w$ contains $w(i)w(j) y_s y_{m+1}$ that gives the associated pattern $\hpat{2314}$. On the other hand, if $w^{-1}(y_{m+1})> h(j)$, then $w(j)(y_{m+1}-1) y_s y_{m+1}$ gives the associated pattern $\hpat{2314}$.
\end{itemize}
Now let $\overline{v} = w(i)w(j)w(k)y_{m+1}$. It follows that $w^{-1}(y_m) \leq h(j)$, and we may assume that $y_m<w(j)$. In this case if $w^{-1}(y_{m+1})\leq h(j)$, then $w$ contains the subsequence $w(i)w(j)w(k) y_{m+1}$ that gives the associated pattern $\hpat{1324}$. On the other hand, if $w^{-1}(y_{m+1}) > h(j)$, then $w(j)(y_{m+1}-1) y_m y_{m+1}$ gives the associated pattern $\hpat{2314}$.

\textbf{Case $\hpat{1243}$.} 
Suppose that $\overline{w}(i)<\overline{w}(j)<\overline{w}(\ell)<\overline{w}(k)$ and $j\leq h(i)<\ell\leq h(j)$. 
It suffices to consider the subsequence $\overline{v}=w(i)w(j)w(k)y_{m+1}$. 
In this case we have $h(i)<w^{-1}(y_m) \leq h(j)$, and we may assume that $y_m<w(j)$. Note that if $w^{-1}(y_{m+1})\leq h(j)$, then the subsequence $w(i)w(j)w(k) y_{m+1}$ of $w$ gives the associated pattern $\hpat{1243}$. On the other hand, if $w^{-1}(y_{m+1})> h(j)$, then $w(j)(y_{m+1}-1) y_m y_{m+1}$ gives the associated pattern $\hpat{2314}$.

\textbf{Case $\hpat{2134}$.} 
Suppose that $\overline{w}(j)<\overline{w}(i)<\overline{w}(k)<\overline{w}(\ell)$ and $k\leq h(i)<\ell\leq h(k)$.
In this case, we have to consider all of the forms of $\overline{v}$. If $\overline{v}$ is one of the forms $w(i)y_{q+1} y_{s+1} y_{m+1}$ and $\overline{v}=w(i)w(j)y_{s+1} y_{m+1}$, then we have $w^{-1}(y_s)\leq h(i)<w^{-1}(y_m) \leq h(w^{-1}(y_s))$, and we may assume that $y_s<w(i)$. If $w^{-1}(y_{s+1})\leq h(i)$, then $w$ contains the subsequence $w(i) y_s y_{s+1} y_{m+1}$ that gives the associated pattern $\hpat{2134}$. On the other hand, if $w^{-1}(y_{s+1})> h(i)$, then the subsequence $w(i) (y_{s+1}-1) y_s y_{s+1}$ gives the associated pattern $\hpat{2314}$.

Now let $\overline{v}=w(i)w(j)w(k)y_{m+1}$. Then $h(i)<w^{-1}(y_m) \leq h(k)$, and we may assume that $y_m<w(k)$. If $w^{-1}(y_{m+1})\leq h(k)$, then $w$ contains the subsequence $w(i)w(j)w(k) y_{m+1}$ that gives the associated pattern $\hpat{2134}$. On the other hand, if $w^{-1}(y_{m+1})> h(k)$, then $w(k) (y_{m+1}-1) y_m y_{m+1}$ gives the associated pattern $\hpat{2314}$ in $w$.

\textbf{Case $\hpat{1423}$.} 
Suppose that $\overline{w}(i)<\overline{w}(k)<\overline{w}(\ell)<\overline{w}(j)$ and $k\leq h(i)<\ell\leq h(j)$. 
It suffices to consider the subsequence $\overline{v}=w(i)w(j)y_{s+1} y_{m+1}$ or $\overline{v}=w(i)w(j)w(k)y_{m+1}$.
First, let $\overline{v}=w(i)w(j)y_{s+1} y_{m+1}$. Then it follows that $w^{-1}(y_s)\leq h(i)<w^{-1}(y_m) \leq h(j)$, and we may assume that $y_s<w(i)$. If $w^{-1}(y_{s+1})\leq h(i)$, then $w$ contains the subsequence $w(i) y_s y_{s+1} y_{m+1}$ that gives the associated pattern $\hpat{2134}$. On the other hand, if $w^{-1}(y_{s+1})> h(i)$, then $w(i)(y_{s+1}-1) y_s y_{s+1}$ gives the associated pattern $\hpat{2314}$ in $w$.

Now let $\overline{v}=w(i)w(j)w(k)y_{m+1}$. Then it follows that $h(i)<w^{-1}(y_m) \leq h(j)$, and we may assume that $y_m<w(k)$. If $w^{-1}(y_{m+1})\leq h(j)$, then $w$ contains the subsequence $w(i)w(j)w(k) y_{m+1}$ that gives the associated pattern $\hpat{1423}$. On the other hand, if $w^{-1}(y_{m+1})> h(j)$, then $w$ contains the associated pattern $\hpat{2314}$ or $\hpat{2413}$. Indeed, if $w^{-1}(1)\leq h(i)$, then $w(i) (y_{m+1}-1) 1 y_{m+1}$ gives $\hpat{2314}$; if $w^{-1}(1)> h(i)$, then $w(i) w(j) 1 y_{m+1}$ gives $\hpat{2413}$.

\textbf{Case $\hpat{2314}$.} 
Suppose that $\overline{w}(k)<\overline{w}(i)<\overline{w}(j)<\overline{w}(\ell)$ and $k\leq h(i)<\ell\leq h(j)$. 
Again, it suffices to consider the subsequence $\overline{v}=w(i)w(j)y_{s+1} y_{m+1}$ or $\overline{v}=w(i)w(j)w(k)y_{m+1}$.
First, let $\overline{v}=w(i)w(j)y_{s+1} y_{m+1}$. Then $w^{-1}(y_s)\leq h(i)<w^{-1}(y_m) \leq h(j)$, and we may assume that $y_m<w(j)$. Accordingly, the subsequence $w(i) (y_{m+1}-1) y_s y_{m+1}$ gives the associated pattern $\hpat{2314}$.

Now we let $\overline{v}=w(i)w(j)w(k)y_{m+1}$, so we have $h(i)<w^{-1}(y_m) \leq h(j)$. We may assume that $y_m<w(j)$. If $w^{-1}(y_{m+1})\leq h(j)$, then  $w$ contains $w(i)w(j)w(k) y_{m+1}$ that gives the associated pattern $\hpat{2314}$. On the other hand, if $w^{-1}(y_{m+1})> h(j)$, then the subsequence $w(j) (y_{m+1}-1) y_m y_{m+1}$ gives the associated pattern $\hpat{2314}$.

\textbf{Case $\hpat{2413}$.} 
Suppose that $\overline{w}(k)<\overline{w}(i)<\overline{w}(\ell)<\overline{w}(j)$ for $j\leq h(i)<k\leq h(j)<\ell\leq h(k)$. 
Again, it suffices to consider the subsequence $\overline{v}=w(i)w(j)y_{s+1} y_{m+1}$ or $\overline{v}=w(i)w(j)w(k)y_{m+1}$. First, let $\overline{v}=w(i)w(j)y_{s+1} y_{m+1}$. Then $h(i)<w^{-1}(y_s)\leq h(j)<w^{-1}(y_{m})\leq h(w^{-1}(y_s))$, and we may assume that $y_m<w(i)$. Accordingly, $w$ contains the subsequence $w(i) w(j) y_s y_{m+1}$ that gives the associated pattern $\hpat{2413}$.

Now let $\overline{v}=w(i)w(j)w(k)y_{m+1}$. Then we have $h(j)<w^{-1}(y_{m})\leq h(k)$, and we may assume that $y_m<w(i)$. If $w^{-1}(y_{m+1})\leq h(k)$, then $w$ contains $w(i)w(j)w(k) y_{m+1}$ that gives the associated pattern $\hpat{2413}$. On the other hand, if $w^{-1}(y_{m+1})> h(k)$, then $w$ contains $w(k) (y_{m+1}-1) 1 y_{m+1}$ that gives the associated pattern $\hpat{2314}$.

Thus we conclude that if $\overline{w}$ contains one of the seven associated patterns, then so does $w$. This completes the proof.
\end{proof}

In the following, we show that the converse of Theorem~\ref{thm:irregular} is also true, thereby characterizing the regularity of $\Gamma_{w,h}$ in terms of pattern avoidance for permutations in $\gen$. Having established the necessary structural properties and size computations, we now proceed to prove this main result by induction on $n$, using the previous lemmas and propositions.

\begin{theorem}\label{thm:regular}
Let $w$ be a permutation in $\gen$ for a Hessenberg function~$h$. The graph $\Gamma_{w,h}$ is regular if $w$ avoids all of the associated patterns $\hpat{2143}$, $\hpat{1324}$, $\hpat{1243}$, $\hpat{2134}$, $\hpat{1423}$, $\hpat{2314}$, and $\hpat{2413}$. 
\end{theorem}

\begin{proof}
Suppose that $w$ avoids all of the associated patterns. We prove that $\Gamma_{w,h}$ is regular by induction on $n$, 
the number of vertices of $\Gamma_{w,h}$ (equivalently, the length of the permutation~$w$). 
For $n=1$, $\Gamma_{w,h}$ is a single vertex graph, and this is certainly regular. 

Suppose that $n \geq2$ and $\Gamma_{w',h'}$ is regular for any Hessenberg function $h'$ on $[n-1]$ and the permutations $w'\in\mathcal{G}_{h'}$ that avoid all of the associated patterns. Now we consider a Hessenberg function $h$ on $[n]$ and a permutation $w\in\gen$ that avoids all the associated patterns. From Theorem~\ref{thm:increasing} it suffices to show that $|E_{w,h}(w)|=|E_{w,h}(w_0)|$. From Lemma~\ref{lem:bottom-right-packed} and Proposition~\ref{prop:size of Ewh} (1), it follows that $w$ is a bottom-packed or a right-packed permutation satisfying that $|E_{w,h}(w)|=|E_{w,h}(\overline{w})|$. Hence we only need to show that  $|E_{w,h}(\overline{w})|=|E_{w,h}(w_0)|$.
By Proposition~\ref{prop:wbar}, we see that $\overline{w}$ is a permutation in $\gen$ for $h$, and $\Gamma_{\overline{w},h}$ is the induced subgraph of $\Gamma_{w,h}$ whose vertex set is $\{u\in[w,w_0] \mid u(n)=1\}$, and 
$$
|E_{w,h}(\overline{w})|=|E_{\overline{w},h}(\overline{w})|+|\{ y_i \in Y(w)\mid h(w^{-1}(y_i))=n \text{~for~} 0\leq i < r \}|\,.
$$
On the other hand, by Proposition~\ref{prop:size of Ewh} (2), we have
$$
|E_{w,h}(w_0)|=|E_{\overline{w},h}(w_0)|+|\{y_i \in Y(w)\mid h(w^{-1}(y_i))=n \text{~for~} 0\leq i \leq r\}|-1\,.
$$
Thus, it suffices to show that $\Gamma_{\overline{w},h}$ is regular. Let $h'$ be the Hessenberg function on $[n-1]$ whose incomparability graph is obtained from the incomparability graph of $h$ by deleting the vertex $n$ and the edges incident to it. Let $w'\in \mathfrak{S}_{n-1}$ be the permutation defined by $w'(i)=\overline{w}(i)-1$ for $i=1,2,\dots,n-1$. As $\overline{w}$ is a permutation in $\gen$, it follows that $w'$ is a permutation in $\mathcal{G}_{h'}$. Moreover, it follows from Proposition~\ref{prop:chain_w} that $w'$ avoids all the patterns. Accordingly, we can apply the induction hypothesis to $w'$ and $h'$ and conclude that $\Gamma_{w',h'}$ is regular. This completes the proof.
\end{proof}

\subsection{Associated patterns for an arbitrary permutation}\label{sec:all}

Now, we consider not only permutations in $\gen$ but also all the permutations in $\Sn{n}$ and provide a set of patterns that characterize the regularity of $\Gamma_{w,h}$. We define a few more patterns that are needed.

\begin{definition}\label{def:pattern5}
Let $h$ be a Hessenberg function on $[n]$. For a permutation $w\in \mathfrak{S}_n$, we say that $w$ contains the associated pattern
\begin{enumerate}
\item $\hpat{25314}$ if $w(\ell)<w(i)<w(k)<w(m)<w(j)$;
\item $\hpat{24315}$ if $w(\ell)<w(i)<w(k)<w(j)<w(m)$;
\item $\hpat{14325}$ if $w(i)<w(\ell)<w(k)<w(j)<w(m)$;
\item $\hpat{15324}$ if $w(i)<w(\ell)<w(k)<w(m)<w(j)$; 
\end{enumerate}
for some $i<j<k\leq h(i) <\ell \leq h(j) <m\leq h(k)$.
\end{definition}

Figure~\ref{fig:addedpatterns} shows the induced subgraph of the incomparability graph of $h$ for the new associated patterns. 

\begin{figure}[ht]
	\begin{center}
	\begin{tikzpicture}[scale=0.55]
		\draw (0,2.4) -- (12,2.4) -- (12,-2.5) -- (0,-2.5) -- (0,2.4);
		\draw (0,1.2) -- (12,1.2);
		\node at (6,1.75) {$\hpat{25314}, ~\hpat{24315}, ~\hpat{14325}, ~\hpat{15324}$};
		\foreach \i/\j in {1/i,2/j,3/k,4/\ell,5/m}{
		\coordinate (\i) at (2*\i,-1);
		\filldraw[color=black] (\i) circle (2pt);
		\node[below] at (\i) {$\j$};}
		\draw[thick] (1) -- (5);
		\draw[thick, bend left=80] (1) edge (3);
		\draw[thick, bend left=80] (2) edge (4);
		\draw[thick, bend left=80] (3) edge (5);
	\end{tikzpicture}
	\end{center}
	\caption{Additional associated patterns and the induced subgraph of $\mathrm{inc}(h)$ on $\{i,j,k,\ell,m\}$.}\label{fig:addedpatterns}
\end{figure}

\begin{remark}\label{remark:pattern5}
 Four patterns in Definition~\ref{def:pattern5} are all the possible ones that appear when we consider the permutations satisfying the following two conditions:
 $$ w(i)<w(k)<w(j) \quad \text{ and }  \quad   w(\ell)<w(k)<w(m)\,. $$
\end{remark}

These new patterns are related to the associated pattern $\hpat{2413}$.

\begin{lemma}\label{lem:2413}
Let $h$ be a Hessenberg function and $w\in\gen$.
If $w$ avoids the associated patterns $\hpat{1243}$, $\hpat{2134}$, and $\hpat{1423}$,
then $w$ contains the associated pattern $\hpat{2413}$ if and only if $w$ contains associated pattern $\hpat{25314}$.
\end{lemma}

\begin{proof}
Let $w$ contain the associated pattern $\hpat{25314}$, i.e., $w(\ell)<w(i)<w(k)<w(m)<w(j)$ for some $i<j<k\leq h(i) <\ell \leq h(j) <m\leq h(k)$. Then $w(\ell)<w(i)<w(m)<w(j)$ and $i<j\leq h(i) <\ell \leq h(j) <m\leq h(\ell)$. Accordingly, $w$ contains the associated pattern $\hpat{2413}$. 

Conversely, 
let $w$ contain $w(i)w(j)w(\ell)w(m)$ that gives the associated pattern $\hpat{2413}$. Consider the following set
$$
X\coloneqq \{w(s) \mid w(i)<w(s)<w(m) \text{ for } j<s<\ell\}\,. 
$$
Suppose that $X$ is empty. Let $y$ be the smallest integer satisfying that $\ell<w^{-1}(y)$ and $w(i)<y$. By Lemma~\ref{lem:y}~(1), there exist $z$ with $i<z<w^{-1}(y)\leq h(z)$ such that $w(i)<w(z)<y$.
Moreover $z<j$ since $X$ is empty and $y$ is the smallest. Accordingly, $w$ contains the subsequence $w(i)w(z)w(j)y$ giving the pattern $\hpat{1243}$, which is a contradiction. Hence $X$ is not empty.

Now we show that $w(i)w(j)w(k)w(\ell)w(m)$ gives the associated pattern $\hpat{25314}$ by taking $w(k)=\max X$. It suffices to show that $k\leq h(i)$ and $h(k)\geq m$. 
Note that if $h(i)< w^{-1}(x)$ for all $x\in X$, then $w$ contains the subsequence $w(i)w(z)w(j)y$ for $y=\min X$ and some $z$ satisfying that $i<z<w^{-1}(y)\leq h(z)$ such that $w(i)<w(z)<y$ by Lemma~\ref{lem:y}~(1). Again, it contradicts the fact that $w$ avoids the associated pattern $\hpat{1243}$. Therefore $h(i)\geq w^{-1}(x)$ for some $x\in X$. Note that if $w^{-1}(x_1)\leq h(i)< w^{-1}(x_2)$ for some $x_1,x_2\in X$, then 
$x_1> x_2$
since $w$ avoids the associated pattern $\hpat{1423}$. 
It follows that
$k\leq h(i)$. In this case, if $h(k)<m$, then $w$ contains $w(k)w(\ell)w(s)w(m)$ for some $s$ satisfying that $\ell<s<m$ and $w(k)<w(s)<w(m)$, which gives the associated pattern $\hpat{2134}$. Thus $h(k)\geq m$, and 
this completes the proof.
\end{proof}

Recall that by Theorems~\ref{thm:irregular} and~\ref{thm:regular}, we can characterize the regularity of 
the graph $\Gamma_{w,h}$ for a permutation $w\in\gen$ as follows: 
 $\Gamma_{w,h}$ is regular if and only if $w$ avoids associated patterns $\hpat{2143}$, $\hpat{1324}$, $\hpat{1243}$, $\hpat{2134}$, $\hpat{1423}$, $\hpat{2314}$, and $\hpat{2413}$.  Owing to Lemma~\ref{lem:2413}, we can replace the pattern $\hpat{2413}$ with a new pattern $\hpat{25314}$ in the characterization of the regularity of $\Gamma_{w,h}$.

\begin{remark}\label{rmk:generators} Let $w$ be a permutation in $\gen$ for a Hessenberg function $h$. Then the following are equivalent:
\begin{enumerate}
    \item $w$ avoids associated patterns $\hpat{2143}$, $\hpat{1324}$, $\hpat{1243}$, $\hpat{2134}$, $\hpat{1423}$, $\hpat{2314}$, and \,\,$\hpat{2413}$.
    \item $w$ avoids associated patterns $\hpat{2143}$, $\hpat{1324}$, $\hpat{1243}$, $\hpat{2134}$, $\hpat{1423}$, $\hpat{2314}$, and\,\, $\hpat{25314}$. 
\end{enumerate}
\end{remark}

Now, we consider all permutations, not only the permutations in $\gen$ for $h$. Recall from Proposition~\ref{prop:h-fixed points} that if $w\not\in\gen$ for a Hessenberg function $h$, then $\Omega_{w, h}^T$ is identified with the set $w\widetilde{w}^{-1}[\widetilde{w}, w_0]$, where $\widetilde{w}$ is the corresponding permutation in $\gen$ of $w$ satisfying \eqref{eq:generator}.
Moreover, Lemma~\ref{lemm:iso} tells us that two graphs $\Gamma_{w,h}$ and $\Gamma_{\widetilde{w},h}$ are isomorphic, which allows us to consider the graph $\Gamma_{\widetilde{w},h}$ of the  corresponding permutation $\widetilde{w}\in\gen$ of $w$ instead of $\Gamma_{w,h}$.

\begin{lemma}\label{lem:allpatterns}
For a Hessenberg function $h$ and a permutation $w$, let $\widetilde{w}$ be the corresponding permutation in $\gen$ of $w$. 
\begin{enumerate}
\item For each associated pattern $p\in \{ \hpat{2143}, \hpat{1324}, \hpat{1243},\hpat{2134}, \hpat{1423}, \hpat{2314}\}$, $w$ contains $p$ if and only if $\widetilde{w}$ contains $p$.
\item 
Let $w$ avoid the associated pattern $\hpat{1324}$. 
Then $w$ contains one of the associated patterns $\hpat{25314}$, $\hpat{24315}$, $\hpat{14325}$, and $\hpat{15324}$ if and only if $\widetilde{w}$ contains the associated pattern $\hpat{25314}$. 
\end{enumerate}
\end{lemma}

\begin{proof}
\begin{enumerate}
\item Suppose that $w$ contains the subsequence $w(i)w(j)w(k)w(\ell)$ that gives an associated pattern $p\in \{ \hpat{2143},  \hpat{1324}, \hpat{1243},\hpat{2134}, \hpat{1423}, \hpat{2314}\}$ for some $i<j<k<\ell$. Due to (\ref{eq:generator}) 
and the definition of the associated patterns, we have
\begin{itemize}
    \item for $p=\hpat{2143}$, $\ell\leq h(i)$;
    \item for $p\in\{\hpat{1324},\hpat{1243},\hpat{1423}\}$, 
    $\widetilde{w}(i)<\widetilde{w}(j),\widetilde{w}(k),\widetilde{w}(\ell)$,
    the position of the second minimum of $p$ is not greater than $h(i)$, and $\ell\leq h(j)$;
    \item for $p\in\{\hpat{2134},\hpat{2314}\}$, 
    $\widetilde{w}(\ell)>\widetilde{w}(i),\widetilde{w}(j),\widetilde{w}(k)$
    and $k\leq h(i)$.
\end{itemize}
Therefore, $\widetilde{w}(i),\widetilde{w}(j),\widetilde{w}(k),\widetilde{w}(\ell)$ are in the same order as $w(i), w(j), w(k), w(\ell)$ for any $p$, i.e., $\widetilde{w}$ contains $p$ if $w$ contains $p$. Since the roles of $w$ and $\widetilde{w}$ are symmetric, the proof is completed. 
\item Let $\widetilde{w}$ contain the associated pattern $\hpat{25314}$, i.e., $\widetilde{w}(\ell)<\widetilde{w}(i)<\widetilde{w}(k)<\widetilde{w}(m)<\widetilde{w}(j)$ for some $i<j<k\leq h(i) <\ell \leq h(j) <m\leq h(k)$. Then from~\eqref{eq:generator}, $w$ satisfies $w(i)<w(k)<w(j)$ and $w(\ell)<w(k)<w(m)$. Therefore, by Remark~\ref{remark:pattern5}, $w$ contains one of the associated patterns $\hpat{25314}$, $\hpat{24315}$, $\hpat{14325}$, and $\hpat{15324}$.

Conversely, let $w$ contain the subsequence $w(i)w(j)w(k)w(\ell)w(m)$ that gives an associated pattern in $P\coloneqq\{\hpat{25314},\hpat{24315}, \hpat{14325}, \hpat{15324}\}$. By (\ref{eq:generator}) and Remark~\ref{remark:pattern5}, the subsequence $\widetilde{w}(i)\widetilde{w}(j)\widetilde{w}(k)\widetilde{w}(\ell)\widetilde{w}(m)$ gives an associated pattern $p \in P$. We show that $p\not\in\{\hpat{24315}, \hpat{14325},\hpat{15324}\}$ because $\widetilde{w}\in\gen$. 
\begin{itemize}
\item[$\bullet$] For $p\in \{\hpat{24315}, \hpat{14325}\}$, let $x$ be the smallest integer satisfying that $\widetilde{w}(j)<x$ and $k<\widetilde{w}^{-1}(x)$. Then $\widetilde{w}^{-1}(x-1)<k$. If $i<\widetilde{w}^{-1}(x-1)$ (respectively, $i>\widetilde{w}^{-1}(x-1)$), then $\widetilde{w}$ contains $\widetilde{w}(i)(x-1)\widetilde{w}(k)x$ (respectively, $\widetilde{w}(i)\widetilde{w}(j)\widetilde{w}(k)x$) that gives the associated pattern $\hpat{1324}$. This contradicts the assumption by (1).
\item[$\bullet$] For $p=\hpat{15324}$, let $x$ be the smallest integer satisfying that $\widetilde{w}(i)<x$ and $k<\widetilde{w}^{-1}(x)$. Then $\widetilde{w}^{-1}(x-1)<k$ and $x\leq w(\ell)$. If $m<\widetilde{w}^{-1}(x)$ (respectively, $m>\widetilde{w}^{-1}(x)$), then $\widetilde{w}$ contains $(x-1)\widetilde{w}(k)\widetilde{w}(\ell)\widetilde{w}(m)$ (respectively, $(x-1)\widetilde{w}(k)x\widetilde{w}(m)$) that gives the associated pattern $\hpat{1324}$. This contradicts the assumption by (1).
\end{itemize}
Thus $p=\hpat{25314}$, as we desired.
\end{enumerate}
\end{proof}

The following is our main theorem.

\begin{theorem}\label{thm:main}
For a Hessenberg function $h\colon [n] \to [n]$, let $w$ be a permutation on $[n]$. The graph $\Gamma_{w,h}$ is regular if and only if $w$ avoids all of the associated patterns $\hpat{2143}$, $\hpat{1324}$, $\hpat{1243}$, $\hpat{2134}$, $\hpat{1423}$, $\hpat{2314}$, $\hpat{25314}$, $\hpat{24315}$, $\hpat{14325}$, and $\hpat{15324}$. 
\end{theorem}

\begin{proof}
Let $\widetilde{w}$ be the corresponding permutation in $\gen$ of $w$. Let us denote the associated pattern sets,
\begin{align*}
&A\coloneqq\{\hpat{2143}, \hpat{1324}, \hpat{1243}, \hpat{2134}, \hpat{1423}, \hpat{2314} \},\\
&B\coloneqq\{\hpat{2143}, \hpat{1324}, \hpat{1243}, \hpat{2134}, \hpat{1423}, \hpat{2314}, \hpat{2413}\},\\
&C\coloneqq\{\hpat{2143}, \hpat{1324}, \hpat{1243}, \hpat{2134}, \hpat{1423}, \hpat{2314}, \hpat{25314}, \hpat{24315}, \hpat{14325}, \hpat{15324}\}\,.
\end{align*}

If $w$ contains an associated pattern in $A$, then so does $\widetilde{w}$ by Lemma~\ref{lem:allpatterns} (1). For a permutation $w$ that avoids all of the associated patterns in $A$, if $w$ contains an associated pattern in $\{\hpat{25314}, \hpat{24315}, \hpat{14325}, \hpat{15324}\}$, then $\widetilde{w}$ contains the associated pattern $\hpat{2413}$ by Lemmas~\ref{lem:2413} and \ref{lem:allpatterns} (2).
Hence $\widetilde{w}$ contains an associated pattern in $B$ if $w$ contains an associated pattern in $C$, and $\Gamma_{w,h}$ is irregular by Lemma~\ref{lemm:iso} and Theorem~\ref{thm:irregular}.

Similarly, if $w$ avoids all of the associated patterns in $C$, then $\widetilde{w}$ avoids all of the associated patterns in $B$, and $\Gamma_{w,h}$ is regular by Lemma~\ref{lemm:iso} and Theorem~\ref{thm:regular}. This completes the proof.
\end{proof}

Once again, by Proposition~\ref{prop:Hessenberg Schubert intersection}, we can provide a necessary condition for $\Omega_{\widetilde{w}}\cap\Hess(S,h)$ ($w\in\Sn{n}$) to be smooth, which is an extension of Theorem~\ref{thm:not-smooth_generator}. 

\begin{theorem}\label{thm:not-smooth}
    For a Hessenberg function $h$, if a permutation $w$ contains any one of the associated patterns  
    $\hpat{2143}$, $\hpat{1324}$, $\hpat{1243}$, $\hpat{2134}$, $\hpat{1423}$, $\hpat{2314}$, $\hpat{25314}$, $\hpat{24315}$, $\hpat{14325}$, and $\hpat{15324}$,  
    then for the corresponding permutation $\widetilde{w}$ in $\gen$ of $w$, the intersection $\Omega_{\widetilde{w}}\cap \Hess(S,h)$ is not smooth.
\end{theorem}

For a permutation $w\in\gen$, the pattern avoidance condition that we found is
also an equivalent condition for the smoothness of  $\Omega_{w}\cap\Hess(\overline{S},h)$ for any diagonal matrix $\overline{S}$ with distinct eigenvalues.

\begin{theorem}\cite[Theorem~1.2]{HLP}\label{conj:regular} For a Hessenberg function $h\colon [n] \to [n]$ and $w\in\gen$, if $\Gamma_{w,h}$ is regular, then the intersection $\Ow{w}\cap\Hess(\overline{S},h)$ is smooth  for any diagonal matrix $\overline{S}$ with distinct eigenvalues. 
\end{theorem}

\begin{theorem} \label{thm:final}
    The following statements hold:
\begin{enumerate}
   \item  If $w\in\gen$ avoids all the patterns $\hpat{2143}$, $\hpat{1324}$, $\hpat{1243}$, $\hpat{2134}$, $\hpat{1423}$, $\hpat{2314}$, and \,\,$\hpat{2413}$, then the Hessenberg Schubert variety $\Owh{w}$ is smooth.
  \item If $w\in\Sn{n}$ avoids all the patterns $\hpat{2143}$, $\hpat{1324}$, $\hpat{1243}$, $\hpat{2134}$, $\hpat{1423}$, $\hpat{2314}$, $\hpat{25314}$, $\hpat{24315}$, $\hpat{14325}$, and $\hpat{15324}$, then the Hessenberg Schubert variety $\Owh{w}$ is smooth.
\end{enumerate}
\end{theorem}

\begin{proof} 
    The second statement follows from Proposition~\ref{prop:intersection_smooth}, Lemma~\ref{lemm:iso}, and Theorem~\ref{conj:regular}.
\end{proof}

Figure~\ref{fig:diagram} shows the relationship between our results and Theorem~\ref{conj:regular} 
related to the regularity of $\Gamma_{w,h}$ and the smoothness of $\Owh{w}$, 
and the following question remains open.
\begin{problem}
Find an equivalent pattern avoidance condition for $\Owh{w}$ to be smooth.
\end{problem}

\begin{figure}[htb]
    \centering
    \begin{tikzpicture}[>=latex']
        \tikzset{block/.style= {draw, rectangle, align=center,minimum width=2cm,minimum height=1cm},
        }
        \node [block]  (start) {$\Gamma_{\widetilde{w},h}$: regular};
        \node [coordinate, below = 0.5cm of start] (aa){};
        \node [right = 0.1cm of aa] (ab){Lem.~\ref{lemm:iso}};

        \node [block, below = 1cm of start] (A1){$\Gamma_{w,h}$: regular};
        \node [coordinate, right = 1.2cm of start] (ca){};
        \node [above = 0.1cm of ca] (cb){Thm.~\ref{conj:regular}};

        \node [block, right = 2.5cm of start, text width=4cm] (B1){ $\Omega_{\widetilde{w}}\cap\Hess(\overline{S},h)$: smooth for any $\overline{S}$};
        \node [coordinate, right = 0.75cm of B1] (da){};

        \node [block, right = 1.85cm of B1] (B2){$\Omega_{\widetilde{w},h}$: smooth};
        \node [coordinate, below = 0.1cm of ca] (ba){};
        \node [below = 0.1cm of ba] (bb){Prop.~\ref{prop:Hessenberg Schubert intersection}};

        \node [coordinate, right = 1cm of B1] (ka){};
        \node [below = 0.4cm of ka] (kb){Prop.~\ref{prop:intersection_smooth}};

        \node [block, below = 1cm of B2] (A2){$\Omega_{w,h}$: smooth};
        \node [coordinate, right = 4.5cm of A1] (ea){};

        \draw[double,<->] (start) -- (A1);
        \draw[double,<-]  ([yshift= 5pt] B1.west) -- ([yshift= 5pt] start.east);
        \draw[double,<-]  ([yshift= -5pt] start.east) -- ([yshift= -5pt] B1.west);
        \draw[double,->] (B1) -- (B2);
        \draw[double,->] (B1) -- (A2);
    \end{tikzpicture}
    \caption{Relations between the regularity of $\Gamma_{w,h}$ and the smoothness of~$\Owh{w}$}\label{fig:diagram}
\end{figure}

\section*{Acknowledgements}
We are deeply grateful to the referees for their careful reading and valuable comments.


\end{document}